\documentclass[11pt]{article}
\usepackage{multirow}
\usepackage{booktabs}
\usepackage{threeparttable}
\usepackage{amsmath,amsthm,amsfonts,amssymb,amscd}
\usepackage[scr]{rsfso}
\usepackage{algorithm,algorithmic}
\usepackage[utf8]{inputenc} 
\usepackage{array} 
\usepackage{stackrel}
 \usepackage{amsmath}
\usepackage{enumerate} 
\usepackage{color}
\usepackage{float}
\usepackage{soul}
\usepackage{subcaption}
\usepackage{graphicx}
\usepackage{mathpazo}
\usepackage{fancyhdr}
\usepackage{empheq}
\usepackage{hyperref}
\hypersetup{
  colorlinks=true,
  linkcolor=red,
  citecolor=blue,
  urlcolor=blue
}
\usepackage[margin=1in]{geometry}
\usepackage{float}
\usepackage{booktabs}

\def\XX{\mathbb{X}}
\def\YY{\mathbb{Y}}

\def\Diag{\mbox {\rm Diag}\,}

\def\Bar{\overline}
\def\ra{\rangle}
\def\la{\langle}
\def\ve{\varepsilon}
\def\B{\mathbb{B}}
\def\h{\hfill\Box}
\def\R{\mathbb{R}}
\def\ox{\bar{x}}

\def\oz{\bar{z}}

\def\ov{\bar{v}}

\def\conv{\mbox{\rm conv}\,}
\def\cone{\mbox{\rm cone}\,}
\def\ri{\mbox{\rm ri}\,}

\def\Im{\mbox{\rm Im}\,}
\def\para{\mbox{\rm par}\,}
\def\gph{\mbox{\rm gph}\,}
\def\epi{\mbox{\rm epi}\,}

\def\Span{\mbox{\rm span}\,}
\def\aff{\mbox{\rm aff}\,}

\def\dom{\mbox{\rm dom}\,}
\def\Ker{\mbox{\rm Ker}\,}

\def\sign{\mbox{\rm sign}\,}

\def\cl*co{\mbox{\rm cl}^*\mbox{\rm co}\,}

\def\cl{\mbox{\rm cl}\,}

\def\h{\hfill\triangle}

\def\oR{\Bar{\R}}

\def\lm{\lambda}

\def\al{\alpha}

\def\N{\mathbb{N}}

\def\hs7{\hspace*{7pt}}

\def\Id{\mathbb{I}}

\DeclareMathOperator{\prox}{\mathbf{prox}}
\renewcommand{\theequation}{\thesection.\arabic{equation}}

\def\h{\hfill\Box}
\def\kk{\kappa}
\DeclareMathOperator*{\argmin}{arg\,min}
\usepackage{url}
\begin{document}

\newtheorem{Theorem}{Theorem}[section]
\newtheorem{Conjecture}[Theorem]{Conjecture}
\newtheorem{Proposition}[Theorem]{Proposition}
\newtheorem{Remark}[Theorem]{Remark}
\newtheorem{Lemma}[Theorem]{Lemma}
\newtheorem{Corollary}[Theorem]{Corollary}
\newtheorem{Definition}[Theorem]{Definition}
\newtheorem{Example}[Theorem]{Example}
\newtheorem{Fact}[Theorem]{Fact}
\newtheorem*{pf}{Proof}
\renewcommand{\theequation}{\thesection.\arabic{equation}}
\normalsize
\normalfont
\medskip
\def\endproof{$\h$\vspace*{0.1in}}

\title{\bf Nonsmooth Newton methods with effective subspaces for polyhedral regularization}
\date{}
\author{TRAN T. A. NGHIA\footnote{Department of Mathematics and Statistics, Oakland University, Rochester, MI 48309, USA; email: nttran@oakland.edu}
\and NGHIA V. VO\footnote{Department of Mathematics and Statistics, Oakland University, Rochester, MI 48309, USA; email: nghiavo@oakland.edu}\and  KHOA V. H. VU\footnote{Department of Electrical and Computer Engineering, Oakland University, Rochester, MI 48309, USA; email: khoavu@oakland.edu}}

\maketitle
\vspace{-0.2in}
\begin{abstract}
We propose several new nonsmooth Newton methods for solving convex composite optimization problems with polyhedral regularizers, while avoiding the computation of complicated second-order information associated with these functions. Under a tilt-stability condition at the optimal solution, the proposed methods achieve the quadratic convergence rates expected of Newton schemes. Numerical experiments on Lasso, generalized Lasso, OSCAR-regularized least-squares problems, and an image super-resolution task illustrate the applicability of the framework and demonstrate the acceleration obtained by the proposed Newton-type update over the corresponding first-order schemes, together with favorable performance relative to recently developed nonsmooth second-order methods on the tested instances.
\end{abstract}

\noindent{\bf Mathematics Subject Classification (2020)}. 49J52, 49J53, 49M15, 65K10, 90C25
\vspace{0.1in}

\noindent{\bf Keywords:} Nonsmooth Newton methods, Polyhedral regularization, Generalized Hessian, Variational Analysis, Convex Optimization

\section{Introduction}
\setcounter{equation}{0}

In this paper, we mainly design nonsmooth Newton methods for solving the following composite optimization problem
\begin{equation}\label{p:CPintro}
\min_{x\in \XX}\qquad \varphi(x):=f(x)+g(x),    
\end{equation}
where $\XX$ is a Euclidean space, $f:\XX\to \R$ is a convex function that is continuously twice differentiable, and $g:\XX\to \oR$, where $\oR:=\R\cup\{+\infty\}$, is a convex lower semicontinuous {\em polyhedral} (possibly nonsmooth) penalty function. Problems of the form \eqref{p:CPintro} cover a broad variety of applications in machine learning, statistics, signal processing, image analysis, and optimal control. In this setting, one typically combines a data-fidelity term with a regularization term that encodes prior structural information on the solution. 

To list a few important optimization problems of the format \eqref{p:CPintro}, we have the classical Least Absolute Shrinkage and Selection Operator \cite{Tibshirani96} (in brief, Lasso) problem well-studied in statistics, signal processing, and sparse optimization, in which the nonsmooth regularizer is the $\ell_1$ norm, $g(x) = \|x\|_1 = \sum_{i=1}^n |x_i|$, $x\in \R^n$, that is certainly a polyhedral convex function. The $\ell_1$ regularizer is also used in different problems, including the  $\ell_1$ regularized logistic regression widely employed for binary classification and high-dimensional feature selections when the fidelity term $f(x)$ is the {\em logistic loss}. Another significant class of polyhedral regularizer is the {\em 1D total variation} (TV) semi-norm defined by $g(x) = \sum_{i=1}^{n-1} |x_{i} - x_{i+1}|$, $x\in \R^n$, which promotes piecewise constant solutions and is significant in signal reconstruction. Beyond these classical examples, several structured polyhedral regularizers have been proposed to capture richer patterns. The fused Lasso \cite{Tibshirani05} augments the $\ell_1$ penalty with successive differences, $g(x) = \lambda_1 \|x\|_1 + \lambda_2 \sum_{i=1}^{n-1} |x_{i} - x_{i+1}|$ with parameters $\lm_1,\lm_2>0$, encouraging both sparsity and smoothness in ordered features such as time-series or genomic data. The clustered Lasso \cite{She10} penalizes pairwise differences between variables, $g(x) = \lambda_1 \|x\|_1 + \lambda_2 \sum_{i < j} |x_i - x_j|$, thereby promoting simultaneous sparsity and grouping of predictors with similar effects. The OSCAR penalty \cite{Bondell08}, $g(x) = \lambda_1 \|x\|_1 + \lambda_2 \sum_{i < j} \max\{|x_i|;\,|x_j|\}$, encourages both sparsity and grouping of predictors with equal magnitude. Finally, the $\ell_1 + \ell_\infty$ regularizer \cite{WuShenGeyer09}, $g(x) = \lambda_1 \|x\|_1 + \lambda_2 \|x\|_\infty$, induces structured sparsity by combining element-wise shrinkage with block-wise suppression. 

With its wide range of applications, there are numerous iterative strategies for solving problem \eqref{p:CPintro}, which can be classified into first-order and second-order methods. Proximal first-order methods such as the Proximal Gradient Method (a.k.a. Iterative Shrinkage-Thresholding Algorithm (ISTA)), the Fast Proximal Gradient Method (a.k.a.  Fast Iterative Shrinkage-Thresholding Algorithm (FISTA)), the Alternating Direction Method of Multipliers (ADMM), the Primal-Dual Algorithm, the coordinate-descent method (CD) and their variants are the most widely used; see, e.g., \cite{BeckTeboulle09,Boyd11,ChambollePock11, bertrand2022beyond} for a comprehensive analysis of these algorithms. They rely on inexpensive proximal mappings to handle the nonsmooth term $g$ and enjoy global convergence guarantees with sublinear rates or linear rates under stronger assumptions on the model \eqref{p:CPintro}. 

Nevertheless, first-order methods may exhibit slow convergence in high-accuracy regimes. This motivates the development of Newton-type methods for solving \eqref{p:CPintro} that can exploit second-order information and usually yield a faster convergence rate.

Adapting the original Newton method to nonsmooth composite optimization is highly nontrivial, since the classical Hessian is undefined on the nonsmooth function $g$. Several major approaches have emerged. One of which is the development of methods that leverage second-order information without the reliance on proximal mappings $\operatorname{prox}_{\alpha g}$ in closed form or can be computed
efficiently. Namely
\(\mathcal{VU}\)- and bundle-type methods for general convex optimization \cite{MG05},
gradient-sampling methods \cite{Burke20,G22}, active-set and manifold-based methods for composite
nonsmooth minimization \cite{BIM23}, and nonsmooth quasi-Newton methods such as BFGS-type
schemes \cite{LO13}. These methods form a different branch of nonsmooth optimization, aiming to handle
settings in which the proximal mapping is unavailable or expensive. Another well-known direction is {\em Semismooth} Newton methods in \cite{Mifflin77,QiSun93,FacchineiPang03} that exploit the semismoothness of proximal mappings and Clarke generalized derivatives, and have been widely applied to complementarity problems and variational inequalities. While effective, establishing their local quadratic convergence typically relies on additional assumptions--most notably, the availability of a computable proximal mapping and the semismoothness property of this mapping. Moreover, finding one element of the Clarke generalized derivative of $\mathrm{prox}_g$ can be a very difficult task in many complicated situations. For instance, for the total variation regularizer $g(x)=\|Dx\|_1$, this requires linearizing an implicit optimality system with unknown active sets and possible degeneracies, making the task technically demanding. A different line of work is the {\em coderivative-based} Newton methods recently developed in \cite{MordukhovichSarabi21,KhanhMordukhovichPhat24,KhanhMordukhovichPhat25}; a systematic exposition can be found in \cite[Chapter~9]{Mordukhovich24}. This framework also provides a natural extension of Newton steps by using the so-called {\em generalized Hessian} (a.k.a. the {\em second-order limiting 
subdifferential}) \cite{Boris92,PoliquinRock98} on the function $g$ and the key notion of {\em semismoothness}$^*$ introduced by Gfrerer and Outrata \cite{GfrererSIOPT21}. The complexity of superlinear convergence of these methods shows evidence of their efficiency in solving many optimization problems \cite{KhanhMordukhovichPhat24,KhanhMordukhovichPhat25}. However, computing the generalized Hessian explicitly  is a challenge in practice that makes finding Newton directions based on these frameworks to be limited in some special and simple cases of $g$. Another direction has been developed recently in  \cite{GfrererCOA25,Gfrerer25,GfrererSIOPT21,Gfrerer22} introducing the so-called  {\em Subspace Containing Derivative} (SCD) Newton method, which instead extracts second-order information and some special subspaces from the {\em graphical derivative} of the subdifferential mapping $\partial g$. \cite{Gfrerer22} showed that the SCD Newton method is a special case of the coderivative-based one, while the former is computationally simpler. Again, the superlinear convergence is obtained for this method under reasonable local conditions at the optimal solutions $\ox$, including the SCD semismoothness$^*$ \cite[Definition~3.4]{Gfrerer25} of $\partial g$ at $(\ox,-\nabla f(\ox))$ and the fulfillment of their regularity condition \cite[Equation~(3.9)]{Gfrerer25}. The former is ensured if $g$ is convex and $\gph \partial g$ is a closed subanalytic set, while the latter follows from the metric regularity of $\partial (f+g)$ at $(\ox,0)$. Impressive experiments for different frameworks related to $\ell_1$ regularizers conducted in \cite{GfrererCOA25,Gfrerer25} clearly show the power of this method.  However, it seems that the computation of subspace containing derivatives on more complex regularizers $g$ is still complicated and abstract. Another geometric perspective for nonsmooth Newton method for solving \eqref{p:CPintro} for $\mathcal{C}^2$-partly smooth regularizers is exemplified in \cite{BIM23,LewisWylie2021}, where proximal-gradient iterations are accelerated by switching to Riemannian Newton steps once the active manifold is identified. This approach yields quadratic local convergence but requires a relative-interior non-degeneracy condition to make the active manifold to be identified at those points close to the optimal solution in question. However, this relative-interior non-degeneracy condition may fail in certain boundary configurations, namely when $-\nabla f(\ox)$ lies on a lower-dimensional face of $\partial g(\ox)$. For instance, in $\ell_1$-type regularization, this occurs when a zero primal coordinate has
a subgradient component at $\pm 1$, which may hinder stable identification of the active manifold near the solution; for additional details, see Remark~\ref{rem:compare-manifold} and Example~\ref{ex:vary-subspace}.

{\bf Our contribution.} In this paper, we develop  Newton-type methods for problem \eqref{p:CPintro} that are both theoretically rigorous and computationally implementable for convex composite
optimization problems with polyhedral regularizers. 
At each iteration, we construct an {\em effective subspace}, defined as the parallel space of the subdifferential of the conjugate penalty function, and restrict the Newton step to this subspace.  In particular, our main Newton method is designed as follows. 

\begin{algorithm}[H]
\caption{Newton method with effective subspaces}
Input: $x_0\in \mathbb{X}, \eta > 0$.
\begin{algorithmic}[1]\label{algo:poly0}
  \STATE Find $(y_k,z_k)\in \gph \partial g$ satisfying 
  \begin{equation}\label{eq:bound-condI}
      \|z_k - \oz\| + \|y_k - \ox\|\le \eta \|x_k-\ox\|
  \end{equation}
  with $\oz:=-\nabla f(\ox)$.
  \STATE Define the effective subspace $L_k:={\rm par}\, \partial g^*(z_k)$ and find a Newton direction $d_k\in L_k$ solving the following quadratic optimization problem
    \begin{equation*}
  \min_{d\in L_k}\quad \frac{1}{2}\la \nabla^2 f(y_k)d,d\ra-\la z_k+\nabla f(y_k),d\ra.
  \end{equation*}
  \STATE Update $x_{k+1}=y_k-d_k$. 
\end{algorithmic}
\end{algorithm}

In this setup, $\ox$ is an optimal solution of \eqref{p:CPintro} and $\partial g$ stands for the convex subdifferential of $g$. The typical condition~\eqref{eq:bound-condI} originated in \cite{GfrererSIOPT21} gives us flexible choices of $(y_k,z_k)$ without determining the specific $\eta$. For example, when the proximal mapping of $g$ is available, they can be chosen from the ISTA iterations (or FISTA iterations with some modifications); see, e.g., our Section~\ref{sec:polyhedral} for further details. The main contribution in our paper is the Newton step of finding Newton direction $d_k$. {Thus, bundle-type and \(\mathcal{VU}\)-methods \cite{MG05}, gradient-sampling methods \cite{Burke20,G22},
active-set and manifold-based methods \cite{BIM23}, and nonsmooth BFGS-type methods \cite{LO13} belong to
a different algorithmic regime from the one considered in this paper. These
methods are designed for broader nonsmooth settings and do not necessarily use
\(\operatorname{prox}_{\alpha g}\) as a basic computational primitive. By contrast,
our paper focuses on the structured composite model \eqref{p:CPintro} with polyhedral
regularizers, where the first-order geometry of \(g\) can be exploited explicitly.} The {\em effective subspace} $L_k$, which is the parallel space of the subdifferential of the conjugate function $g^*$ at $z_k$, is computed only based on the first-order information. {At first sight}, this step looks similar with the corresponding ones in \cite{BIM23,LewisWylie2021} with active manifolds and \cite{GfrererCOA25,Gfrerer25} with subspaces containing derivative. However, even in the simple case of $\ell_1$ regularizer, our effective subspaces are different from the active subspaces in \cite{BIM23,LewisWylie2021} that actually depend on $y_k$ instead of $z_k$ as designed in our algorithm. We will show later  that our algorithm is a particular case of the SCD Newton method developed in \cite{GfrererCOA25,Gfrerer25} and also the coderivative-based Newton method \cite{KhanhMordukhovichPhat25} for two special classes of regularizers at which $g$ is either an indicator function or a support function of a polyhedral convex set.  
Nevertheless, we have an explicit formula of choosing the effective subspace $L_k$ that is fully computable in many complex polyhedral regularizers aforementioned; see also Section~\ref{sec:exam} and \ref{sec:numerical} for further details.  

Moreover, we prove that this scheme can achieve the quadratic convergence as desired for Newton methods without the non-degeneracy assumption in \cite{BIM23,LewisWylie2021}, while other papers  \cite{GfrererCOA25,Gfrerer25, GfrererSIOPT21,KhanhMordukhovichPhat25,KhanhMordukhovichPhat24, MordukhovichSarabi21} in this direction only obtain superlinear convergence for their Newton methods.
Our analysis for quadratic convergence of our algorithm is grounded in the landmark notion of {\em tilt-stability} of minimizers introduced by Poliquin and Rockafellar \cite{PoliquinRock98}, a concept that characterizes Lipschitz continuity of solution mapping via a {\em tilt/linear perturbation} on the objective function. In convex settings,  this condition at the minimizer $\ox$ is equivalent to some other conditions on {\em (strong) metric regularity} of the subdifferential mapping also used in \cite{GfrererCOA25,KhanhMordukhovichPhat25} for convergence analysis.

It is important to emphasize that the convergence analysis of the proposed method is local in nature. In particular, the quadratic convergence established below holds only when the initial point is sufficiently close to a tilt-stable minimizer $\ox$. As is typical for Newton-type methods with unit stepsizes, no global convergence guarantee can be expected when the algorithm is initialized far from the solution, and divergence may occur outside a neighborhood of $\ox$. We also note that globalization strategies for nonsmooth Newton-type methods have been investigated in the literature via different approaches. One approach is based on line-search procedures applied after manifold identification; see, e.g., \cite{BIM23}. In this framework, once a suitable manifold is identified, the objective function becomes smooth when restricted to this manifold, which allows the use of classical Armijo-type line-search techniques in a Riemannian setting to ensure descent while preserving fast local convergence. Another approach relies on merit functions such as the Forward--Backward Envelope (FBE); see, e.g., \cite{Gfrerer25,KhanhMordukhovichPhat25,KhanhMordukhovichPhat24}, where a globally smooth surrogate is constructed to guide the iteration and ensure convergence from arbitrary starting points. We believe that both types of strategies can be incorporated into our framework, although such developments are beyond the scope of the present paper and are part of our ongoing investigation.

In our paper, an important part is the exact computation of the effective subspace $L_k$ for several classes of polyhedral regularizers, including the support functions of convex sets containing many polyhedral (semi-)norms in $\R^n$ such as the $\ell_1$ norm, the $\ell_\infty$ norm, the sorted $\ell_1$ norm and the 1D total variation semi-norm. We also obtain a general formula for computing this effective subspace of composite convex functions and apply it to evaluate the case of 1D total variation semi-norm mentioned above. Finally, we validate the theoretical findings with numerical experiments on synthetic and image datasets, showing that our Newton variants dramatically accelerate convergence for ISTA/FISTA and can outperform some other second-order algorithms such as those developed recently in \cite{GfrererCOA25,KhanhMordukhovichPhat24,LST18} in solving several problems, including Lasso, $\ell_\infty$-regularized problem, OSCAR, and 1D total variation problem.



\section{Preliminaries}\label{sec:prelim}
\setcounter{equation}{0}

Throughout this paper, we denote by $\XX$ a finite-dimensional Euclidean space endowed with the inner product $\la\cdot,\cdot\ra$ and the corresponding Euclidean norm $\|\cdot\|$. For a subspace $S \subseteq \XX$, $P_{S}$ stands for the {\em orthogonal projection} onto $S$ and $S^{\bot}$ is the {\em orthogonal complement} of $S$. Given a nonempty set $C\subseteq \XX$, the {\em span}, {\em affine hull}, {\em convex conic hull}, {\em convex hull}, and {\em closure} of $C$ are denoted, respectively, by $\Span C$, $\aff C$, $\cone C$, $\conv C$, and $\cl C$.  For a linear operator $A:\XX\to \YY$, we write $\Ker A$ and $\Im A$, respectively, for the {\em kernel/null space} and {\em image/range} of $A$. The set $\mathbb{B}_{\varepsilon}(\bar{x})$ is  the {\em closed ball} in $\XX$ centered at $\bar{x}\in \XX$ with radius $\varepsilon>0$. 
\subsection{Tools from Convex Analysis}
Let $C$ be  a closed convex set in $\XX$. The {\em support function} of $C$ is given by 
\begin{equation}\label{eq:support}
\sigma_C(v) := \sup_{x\in C}\,\la v,x \ra \quad \forall v\in \XX.
\end{equation}
The topological {\em relative interior} of $C$ is the set 
\[
\ri C := \{x \in \XX \mid \exists\, \varepsilon >0: \B_{\varepsilon}(x)\cap\aff C \subset C \}.
\]
The {\em normal cone} to $C$ at $x \in C$ is defined as 
\begin{equation}\label{def:N}
N_{C}(x) := \{v \in \XX \mid \la v,y - x\ra \le 0, \forall y \in C  \}.
\end{equation}
Let us recall the \emph{tangent cone} to the convex set $C$ at $x\in C$ here
\begin{equation}\label{eq:tangent}
T_C (x):= \cl(\cone(C-x)),
\end{equation}
which is a closed convex cone. Indeed, the {\em polar} of the tangent cone $T_C(x)$ is the normal cone $N_C(x)$ defined above and vice versa due to the convexity of $C$.

One of the central definitions throughout the paper is the {\em parallel subspace} of $C\subseteq \XX$, which is given by 
\begin{equation}\label{def:Par}
 {\rm par}\, C:= \Span(C - x) \quad \text{ for all }\quad  x\in C.   
\end{equation}
Note that the parallel subspace of any convex cone $K\subset \XX$ can be calculated as 
\[ {\rm par}\, K = \Span K = K - K.
\]

{For a proper, lower semicontinuous, convex function $\varphi:\XX\to \oR$ with $\,\,\oR:=\R\cup\{+\infty\}$}, the {\em effective domain} and the {\em epigraph} of $\varphi$ are defined, respectively, by
\[
\dom \varphi:= \{x \in \XX \mid \varphi(x) < \infty \} \quad \text{and} \quad \epi \varphi:= \{(x,t) \in \XX \times \R \mid \varphi(x) \le t \}.
\]
The \emph{convex subdifferential} of \(\varphi\) at \(x\in \dom\varphi\) is defined by
\begin{equation}\label{def:convex-subdiff}
\partial \varphi(x):=\big\{v\in \XX \mid \varphi(y)\ge \varphi(x)+\la v,y-x\ra,\ \forall y\in \XX\big\}.
\end{equation}
An equivalent characterization of the subdifferential in terms of the normal cone to the epigraph is
\begin{equation}\label{eq:subdiff-normal-epi}
\partial \varphi(x)=\big\{v\in \XX \mid (v,-1)\in N_{\mathrm{epi}\, \varphi}(x,\varphi(x))\big\}
\quad \text{for all }\quad x\in \dom\varphi.
\end{equation}
The normal cone $N_C(x)$ in \eqref{def:N} is indeed the subdifferential \eqref{def:convex-subdiff} of the {\em indicator function} $\delta_{C}(\cdot)$ at $x$, which is defined by $\delta_{C}(x) = 0$ whenever $x \in C$, and $\infty$ otherwise.

The \emph{Legendre-Fenchel conjugate} of $\varphi$ is the convex function $\varphi^{*}: \XX \to \Bar{\R}$ defined by
\[
\varphi^*(v) := \sup_{x\in \XX}\, \{\la v, x\ra - \varphi(x)\}\ \text{ for all }\ v\in \XX.
\]
When $\varphi(x)=\delta_C(x)$ is the indicator function to a convex set $C$, its conjugate $\delta_C^*$ is exactly the support function $\sigma_C$ \eqref{eq:support} to $C$. 

A fundamental result in Convex Analysis is the Fenchel-Young's identity occasionally used in our paper, which states that 
\begin{equation}\label{def:fenyou}
    v \in \partial \varphi(x) \quad \text{if and only if} \quad \varphi(x) + \varphi^{*}(v) = \la v,x \ra .
\end{equation}

Let us recall next the {\em proximal operator} of a convex function $\varphi$
\begin{equation}\label{def:prox}
\prox_{\alpha \varphi}(x) := \argmin_{y\in \XX}\left\{\varphi(y) + \frac{1}{2\alpha}\|y-x\|^2 \right\} \quad \forall x\in \XX, \forall \al>0.
\end{equation}
The proximal mapping $\prox_{\alpha \varphi}:\XX\to \XX$ associated with a convex function $\varphi$ is a single-valued mapping with full domain. Moreover, it can be expressed equivalently as the {\em resolvent operator} $(\Id + \alpha\partial \varphi)^{-1}$, where $\Id:\XX\to \XX$ is the {\em identity operator}. By the first-order optimality condition, we have
\[
y=\prox_{\alpha\varphi}(x)
\quad\text{if and only if}\quad
0\in \partial\varphi(y)+\frac{1}{\alpha}(y-x),
\]
which, after rearranging, yields the equivalent characterization
\begin{equation}\label{eq:res}
y=\prox_{\alpha\varphi}(x)
\quad\text{if and only if}\quad
\frac{x-y}{\alpha}\in \partial\varphi(y).
\end{equation}
\subsection{Tools from Variational Analysis}
In this section, we mainly follow the notations from the classical monographs \cite{Mordukhovich06,Mordukhovich24,Rockafellar98} on Variational Analysis. Given a closed (possibly nonconvex) set $\Omega\subset \XX$ and $\ox\in \Omega$. The {\em regular normal cone} (a.k.a. Fr\'echet normal cone) to $\Omega$ at $\ox$ is defined by 
\begin{equation*}
    \widehat{N}_\Omega (\ox):=\left\{v\in \XX\;\Bigg|\;\limsup_{x\xrightarrow{\Omega}\ox}\dfrac{\la v, x - \ox\ra}{\|x-\ox\|} \le 0\right\},
\end{equation*}
where $x\xrightarrow{\Omega}\ox$ means that $x\in \Omega$ and $x\to \ox$. The {\em limiting normal cone}  (a.k.a. Mordukhovich normal cone) to $\Omega$ at $\ox\in \Omega$ is given by 
\begin{equation}\label{def:limit-cone}
    N_{\Omega}(\ox) := \big\{v\in \XX \mid \exists\,  \{x_k\}\xrightarrow{\Omega}\ox, \{v_k\}\to v \text{ such that }v_k \in \widehat{N}_{\Omega}(x_k)\; \forall k\in \N\big\}.
\end{equation}
Both the regular and limiting normal cones return exactly to the convex normal cone introduced in \eqref{def:N} when the set $\Omega$ is convex, but they are different in general.

For a set-valued mapping $F: \mathbb{X} \rightrightarrows \mathbb{X}$ from $\XX$ to itself, we associate with it the {\em graph} given by
\[
\gph F := \{(x,y) \in \mathbb{X} \times \mathbb{X} \mid y\in F(x)\}.
\]
The {\em limiting coderivative} of a set-valued mapping $F:\XX \rightrightarrows \XX$ at $(\ox,\ov)\in \gph F$ is defined as
\begin{equation}\label{def:limit-code}
D^{*}F(\ox|\,\ov)(w) := \{z \in \XX \mid (z,-w) \in N_{\mathrm{gph}\,F}(\ox,\ov) \} \quad \text{for all} \quad w \in \XX.
\end{equation}
When $F$ is the convex subdifferential mapping $\partial \varphi$ and $\ov\in \partial \varphi(\ox)$, the limiting coderivative of $\partial \varphi$ at $\ox$ for $\ov$ is known as the {\em generalized Hessian} (a.k.a. {\em second-order limiting subdifferential}) \cite{Boris92,Mordukhovich06,PoliquinRock98} of $\varphi$ at $\ox$ for $\ov$ denoted by:
\begin{equation}\label{def:secondsub}
\partial^{2}\varphi(\ox|\,\ov)(w) := (D^{*}\partial \varphi)(\ox|\,\ov)(w) \quad \text{for all} \quad w \in \XX,
\end{equation}
which plays an important role in our paper and several nonsmooth Newton methods developed recently in \cite{KhanhMordukhovichPhat25,KhanhMordukhovichPhat24,MordukhovichSarabi21}. It is worth noting that if $\varphi$ is convex and twice continuously differentiable around $\ox$ with $\ov = \nabla \varphi(\ox)$, then its generalized Hessian reduces to the regular Hessian mapping
\[
\partial^{2}\varphi(\ox|\, \ov)(w) =  \nabla^2 \varphi(\ox)w \quad \text{for all} \quad w \in \XX.
\]
In this case, the classical Newton method for finding a minimizer $\ox$ of $\varphi$ is designed as: Initiate $x_0\in \XX$ and proceed
\begin{equation}\label{Al:Newton}
x_{k+1}=x_k-d_k\quad \mbox{with}\quad d_k:=[\nabla^2 \varphi(x_k)]^{-1}\nabla \varphi (x_k).
\end{equation}
Since $\ox$ is a (local) minimizer of $\varphi$, the Hessian $\nabla^2\varphi(\ox)$ is always positive semi-definite. The existence of Newton direction $d_k$ requires the Hessian $\nabla^2 \varphi(x_k)$ to be non-singular. To make this happen, it is conventional to assume that the Hessian $\nabla^2\varphi(\ox)$ is positive definite and start with $x_0$ being  sufficiently close to $\ox$. When $\nabla^2 \varphi(x)$ is Lipschitz continuous around $\ox$, it is well-known that  the above sequence $\{x_k\}$ converges {\em quadratically} to $\ox$ in the sense that there exists a constant $C>0$ (independent of $x_0$) such that 
\begin{equation}\label{eq:Qua}
\|x_{k+1}-\ox\|\le C\|x_k-\ox\|^2\quad \mbox{for any }\quad k\in \mathbb{N}. 
\end{equation}

If $\varphi$ is not twice differentiable, designing Newton methods is a huge challenge, as the regular Hessian does not exist anymore. Naturally, replacing the Hessian in \eqref{Al:Newton} by the generalized Hessian \eqref{def:secondsub} is a possibility; see, e.g., {\bf Algorithm~\ref{algo:NewCod}} and  \cite{MordukhovichSarabi21,Mordukhovich24,KhanhMordukhovichPhat24,KhanhMordukhovichPhat25}. Spontaneously, the generalized Hessian $\partial^2 \varphi(\ox|\,0)$ may need to be ``positive definite'' to match the classical theory. This actually leads us to the definition of {\em tilt stability} at a minimizer $\ox$ introduced by Poliquin and Rockafellar in \cite{PoliquinRock98}, which is shown to be equivalent to the positive definiteness of $\partial^2 \varphi(\ox|\,0)$; see \cite[Theorem~1.3]{PoliquinRock98} or  Theorem~\ref{theo:Tilt} below. It is worth mentioning that the general definition and theory of tilt stability \cite{PoliquinRock98,Mordukhovich24} is for nonconvex and nonsmooth functions; here, we slightly modify it for the convex case.  

\begin{Definition}[Tilt stability]\label{Den:Tilt} Given an l.s.c. convex function $\varphi:\mathbb{X} \to \Bar{\R}$, a point $\ox \in \dom \varphi$ is called a \emph{tilt stable} minimizer of $\varphi$ if there exists $\gamma > 0$ such that the mapping 
\begin{equation*}
M_{\gamma}(v) := \argmin\{\varphi(x) - \la v,x \ra \mid x \in \B_{\gamma}(\ox)\} \quad \text{for } \quad v \in \XX
\end{equation*}
is single-valued and Lipschitz continuous on some neighborhood of $0 \in \XX$ with $M_{\gamma}(0) = \{\ox\}.$
\end{Definition}

Tilt-stability also lies behind the theory of nonsmooth Newton methods developed recently in \cite{MordukhovichSarabi21, KhanhMordukhovichPhat24,GfrererCOA25}. The following theorem, a direct consequence of \cite[Theorem~3.5]{BorisNghia15} and \cite[Example~13.30]{Rockafellar98}, gives a complete characterization of tilt stability and will be useful for our later analysis.  

\begin{Theorem}{\rm (A characterization of tilt stability via generalized Hessian)}\label{theo:Tilt} Consider an l.s.c. convex function $\varphi:\mathbb{X}\to \overline{\R}$ with $\ox\in \dom \varphi$ and $0\in \partial \varphi (\ox)$. Then $\ox$ is a \emph{tilt stable minimizer} of $\varphi$ if and only if the generalized Hessian $\partial^2 \varphi$ is positive definite around $(\ox,0)$ in the sense that there exist $\kappa,r>0$ satisfying 
\begin{equation}\label{eq:tilt}
    \la z,w\ra \ge \kappa\|w\|^2 \quad \forall z\in \partial^2 \varphi (x|\,v)(w),\; \forall (x,v)\in \mathbb{B}_{r}(\ox,0)\cap \gph \partial \varphi.
\end{equation}
\end{Theorem}

\section{Newton methods with effective subspaces for polyhedral regularization}\label{sec:polyhedral}
\setcounter{equation}{0}

In this section, we propose some nonsmooth Newton methods for solving the following composite optimization problem
\begin{equation}\label{p:CP}
    \min_{x\in \XX}\quad \varphi(x):=f(x)+g(x),
\end{equation}
where $f:\XX\to \R$ is a convex function that is  twice continuously differentiable and $g:\XX\to \oR$ is a proper piecewise linear/polyhedral convex function. Recall that  a proper function $g$ is piecewise linear/polyhedral and convex if and only if  its epigraph is {\em polyhedral}, i.e., it is the intersection of finitely many closed half-spaces in $\XX\times \R$ \cite[Definition 2.47 and Theorem~2.49]{Rockafellar98}.

Suppose that $\ox$ is an optimal solution of problem~\eqref{p:CP}. To motivate our algorithm, let us recall the following coderivative-based Newton methods in \cite[Algorithm~1 and Algorithm~2]{KhanhMordukhovichPhat25} below.
\begin{algorithm}[H]
\caption{The coderivative-based Newton method}
Input: $x_0\in \mathbb{X}, \eta > 0$.
\begin{algorithmic}[1]\label{algo:NewCod}
  \STATE Find $(y_k,z_k)\in \gph \partial g$ satisfying 
  \begin{equation}\label{eq:bound-cond0}
      \|z_k - \oz\| + \|y_k - \ox\|\le \eta \|x_k-\ox\|,
  \end{equation}
  where $\oz:=-\nabla f(\ox)$.
  \STATE Find a Newton direction $d_k\in \XX$ satisfying
  \begin{equation}\label{eq:NewStep}
  -\nabla^2 f(y_k)d_k+z_k+\nabla f(y_k)\in \partial^2 g(y_k|\, z_k)(d_k).
  \end{equation}
  \STATE Update $x_{k+1}=y_k-d_k$. 
\end{algorithmic}
\end{algorithm}
When $g\equiv 0$, we may choose $\eta=1$,  $y_k=x_k$, and $z_k=\oz=0$ to make \eqref{eq:bound-cond0} holds. In this case, inclusion \eqref{eq:NewStep} reduces exactly to the original Newton step in \eqref{Al:Newton}
\begin{equation}\label{eq:clasNew}
\nabla^2 f(x_k)d_k=\nabla f(x_k).
\end{equation}
Thus, {\bf Algorithm~\ref{algo:NewCod}} is a natural extension of the classical Newton method \eqref{Al:Newton} to problem~\eqref{p:CP}, at which the second-order information of the nonsmooth function $g$ is also involved via the generalized Hessian $\partial^2 g(y_k|\,z_k)$ in ~\eqref{eq:NewStep}; see also \cite[Chapter~9]{Mordukhovich24} for further history and state-of-the-art of this algorithm. {In \cite{KhanhMordukhovichPhat25}, the authors first prove a local convergence result in Theorem~4.1: starting sufficiently close to a stationary point, {\bf Algorithm~\ref{algo:NewCod}} converges Q-superlinearly. They then globalize the method \cite[Theorem~6.1]{KhanhMordukhovichPhat25} using a line-search based on the Forward–Backward Envelope, showing that \cite[Algorithm~4]{KhanhMordukhovichPhat25} eventually behaves like {\bf Algorithm~\ref{algo:NewCod}} and achieves Q-superlinear convergence of the iterates. However, finding the Newton direction by solving inclusion \eqref{eq:NewStep} is a nontrivial step, especially when the function $g$ is complex. Later on, we will show that our {\bf Algorithm~\ref{algo:poly0}} is a particular case of {\bf Algorithm~\ref{algo:NewCod}} when $g$ is either an indicator function or a support function of a polyhedral convex set. But we do not need to compute the complicated generalized Hessian $\partial^2 g$ in \eqref{eq:NewStep} to find the Newton direction other than solving a simple quadratic programming problem. The most crucial part of our paper is the introduction of the effective space $L_k$ that allows us to avoid the generalized Hessian in \eqref{eq:NewStep}. Moreover, we prove that the sequence $\{x_k\}$ converges (locally) quadratically to the solution $x_0$ under mild assumption on $x_0$. Below we  discuss these two algorithms in details.}

For the general problem \eqref{p:CP}, the {\em approximation step} of finding $(y_k,z_k)\in \gph \partial g$ satisfying \eqref{eq:bound-cond0} was initiated  in \cite{GfrererSIOPT21}   and  used recently for several Newton methods in \cite{GfrererCOA25,Gfrerer25,KhanhMordukhovichPhat25,KhanhMordukhovichPhat24}. We may choose the pair
\begin{equation}\label{eq:prox}
y_k = \prox_{\alpha g}(x_k-\alpha\nabla f(x_k))\quad \mbox{and}\quad z_k = \frac{x_k-y_k}{\alpha} - \nabla f(x_k) \quad \mbox{with some }\quad \al>0
\end{equation}
from the proximal gradient method \cite{BeckTeboulle09} to satisfy this condition; see, e.g., \cite{Gfrerer25, KhanhMordukhovichPhat25}. 

The generalized Newton step \eqref{eq:NewStep} has some roots in \cite{MordukhovichSarabi21} and has been fully developed in the recent paper \cite{KhanhMordukhovichPhat25}. In particular, \cite[Theorem~3.11 and Theorem~4.10]{KhanhMordukhovichPhat25} show that if the subdifferential mapping $\partial \varphi$ is {\em metrically  regular} at the optimal solution $\ox$ of problem~\eqref{p:CP} for $0$ in the sense of \cite[Section~3E]{Dont-Rock} and $\partial g$ is {\em semismooth}$^*$ at $(\ox,-\nabla f(\ox))$ in the sense of \cite[Definition~3.1]{GfrererSIOPT21}, then the sequence $\{x_k\}$ converges {\em superlinearly} to $\ox$, i.e., 
\[
\lim_{k\to\infty}\dfrac{\|x_{k+1}-\ox\|}{\|x_k-\ox\|}=0,
\]
provided that the initial $x_0$ is close enough to $\ox$. The results in \cite{KhanhMordukhovichPhat25} work for more general frameworks at which the functions $f$ and $g$ do not need to be convex. For our convex setting in \eqref{p:CP}, the aforementioned metric regularity condition on the subdifferential mapping $\partial \varphi $ at $\ox$ for $0$ is equivalent to the tilt-stability of $\varphi$ at the minimizer $\ox$ in Definition~\ref{Den:Tilt}. Moreover, the subdifferential mapping $\partial g$ of our convex polyhedral function $g$ is also a particular class of semismooth$^*$ multifunctions \cite[Proposition~2.10]{Gfrerer22}. 

The main challenge for {\bf Algorithm~\ref{algo:NewCod}} lies in the computation of the generalized Hessian in practice. Moreover, finding a Newton direction satisfying the inclusion in \eqref{eq:NewStep} seems to be a very nontrivial step. To address these issues, we propose a simple way of finding such $d_k$ without computing the generalized Hessian of $g$. Motivated by the classical Newton step in \eqref{eq:clasNew}, a solution $d_k$ of the following system  
\begin{equation}\label{eq:Motiv}
    -\nabla^2 f(y_k)d_k+z_k+\nabla f(y_k)=0\quad \mbox{and}\quad 0\in \partial^2 g(y_k|\,z_k)(d_k)
\end{equation}
satisfies \eqref{eq:NewStep} obviously. As $g$ is a polyhedral function, so is its Legendre-Fenchel conjugate $g^*$. It follows from \cite[Equation~(4.3)]{BorisRock12} that 
\begin{equation}\label{eq:BoTe}
 0\in \partial^2 g(y_k|\, z_k)(d_k)\quad \mbox{if and only if}\quad d_k\in \partial^2 g^*(z_k|\,y_k)(0)=\para \partial g^*(z_k),
 \end{equation}
which is the parallel space of $\partial g^*(z_k)$. This subspace is referred to as the {\em effective subspace} of $g$ at $y_k$ for $z_k$ in our Newton methods. It actually signifies the null/kernel space of the generalized Hessian of $g$ at $y_k$ for $z_k\in \partial g(y_k)$. The system \eqref{eq:Motiv} turns into a linear system 
\[
-\nabla^2 f(y_k)d_k+z_k+\nabla f(y_k)=0\quad\mbox{and}\quad d_k\in L_k:=\para \partial g^*(z_k),
\]
which has a solution when $\nabla^2 f(y_k)(L_k)=\XX$. This condition may fail, leading to the above linear system may have no solution. As $\nabla^2 f(y_k)$ is positive semidefinite, the first equation hints us that $d_k$ is a solution of the following quadratic optimization problem 
\[
\min_{d\in \XX}\quad \frac{1}{2}\la \nabla^2 f(y_k)d,d\ra-\la z_k+\nabla f(y_k),d\ra.
\]
Since we also need the solution $d_k$ to belong to the subspace $L_k$ as in \eqref{eq:Motiv}, we slightly modify the above optimization problem to
\begin{equation}\label{p:Sub}
  \min_{d\in L_k}\quad \frac{1}{2}\la \nabla^2 f(y_k)d,d\ra-\la z_k+\nabla f(y_k),d\ra,
  \end{equation}
which means that $d_k$ is a solution to the following linear system (instead of the proposed system \eqref{eq:Motiv})
\begin{equation}\label{eq:Incl}
-\nabla^2 f(y_k)d_k+z_k+\nabla f(y_k)\in L_k^\perp \quad\mbox{and}\quad d_k\in L_k.
\end{equation}
Indeed, we show later that this system has a unique solution when $x_0$ is close to the tilt stable minimizer $\ox$. These discussions motivate us to design our Newton method with effective subspaces $L_k$ mentioned in the Introduction.
\begin{algorithm}[H]
\captionsetup{labelformat=empty}
\caption*{{\bf Algorithm 1}: Newton method with effective subspaces}
Input: $x_0\in \mathbb{X}, \eta > 0$.
\begin{algorithmic}[1]
  \STATE Find $(y_k,z_k)\in \gph \partial g$ satisfying \eqref{eq:bound-cond0}.
  \STATE Define the effective subspace $L_k:={\rm par}\, \partial g^*(z_k)$ and find an optimal solution $d_k\in L_k$ of problem \eqref{p:Sub}.
  \STATE Update $x_{k+1}=y_k-d_k$. 
\end{algorithmic}
\end{algorithm}

{\noindent
{\bf Algorithm~\ref{algo:poly0}} should be understood as an abstract local framework rather
than a directly implementable algorithm, since the boundedness condition~\eqref{eq:bound-cond0} is expressed relative to the
unknown solution pair $(\bar{x},\bar{z})$. Its purpose is to identify a
uniform approximation property sufficient for the local quadratic
convergence analysis. More precisely, the convergence result in Theorem~\ref{theo:quad-poly}
applies to every sequence $\{(y_k,z_k)\}
    \subset \gph\partial g$
satisfying~\eqref{eq:bound-cond0} with a common constant
$\eta>0$. Concrete implementable selections satisfying the required
estimate are provided by the ISTA--Newton and FISTA--Newton schemes in
{\bf Algorithms~\ref{algo:poly-ista} and~\ref{algo:poly-fista}}, respectively.}

In Corollary~\ref{coro:compare}, when $g$ is either the indicator or the support function of a polyhedral convex set, we show that our Newton step \eqref{p:Sub} or \eqref{eq:Incl} is indeed a particular case of \eqref{eq:NewStep} and the recent Newton method \cite[Algorithm~4.2]{Gfrerer25} that uses the so-called {\em Subspace Containing} (SC) {\em derivative} of the subdifferential mapping $\partial g$, which is also a second-order structure. The idea of using effective subspaces to find a Newton direction in \eqref{p:Sub} is also close to the latter. However, comparing with other algorithms established recently, it seems that our effective subspaces $L_k$ have more explicit and computable settings, allowing us to avoid the use of more involved generalized second-order objects such as generalized Hessians or subspace containing derivatives. {At first sight}, it seems that the computation of $L_k$ require that of $g^*$, which is in itself another optimization problem. However, we show that the construction of $L_k$ can bypass this step, since 
\begin{equation}\label{eq:Lk-via-g}
\begin{aligned}
    L_k = \mathrm{par}\, \partial g^* (z_k) &= \cone \big(\partial g^* (z_k) - y_k\big) - \cone \big(\partial g^* (z_k) - y_k\big) \\
    &= T_{\partial g^* (z_k)}(y_k) - T_{\partial g^* (z_k)}(y_k) \\
    &= \Span \big[T_{\partial g^* (z_k)}(y_k)\big]\\
    &= \Big(\mathrm{lin}\, \big[T_{\partial g^* (z_k)}(y_k)\big]^* \Big)^\perp\\
    &= \Big(\mathrm{lin}\, \big[T_{\partial g (y_k)}(z_k)\big] \Big)^\perp = \Big(\mathrm{lin}\, \big[\mathrm{cone}\, \big( \partial g(y_k) - z_k\big)\big] \Big)^\perp,
\end{aligned}
\end{equation}
where $\mathrm{lin}\, K:= K\cap (-K)$ denotes the {\em lineality subspace} of a cone $K$, and where the second-to-last identity follows from \cite[Corollary~3.2]{HangJungSarabi24} together with its proof. Nevertheless, we view $L_k$ as $\operatorname{par}\,\partial g^*(z_k)$ for two reasons. 
First, this interpretation preserves the motivating structure of the construction, as reflected in 
\eqref{eq:BoTe}. Second, it is compatible with the parallel-subspace characterization of tilt stability in 
\eqref{eq:TiltPar}, which plays a central role in the subsequent analysis.

Moreover, with these effective subspaces, we are able to prove next that {\bf Algorithm~\ref{algo:poly0}} acquires quadratic convergence around the tilt stable minimizer $\ox$, which is the desirable complexity for Newton methods instead of superlinear convergence obtained in \cite{Gfrerer25,KhanhMordukhovichPhat25}. In addition, our framework is compatible with first-order initialization strategies: in particular, our {\bf Algorithm~\ref{algo:poly0}} and {\bf Algorithm~\ref{algo:NewCod}} from \cite{KhanhMordukhovichPhat25} both employ the approximation step of the form \eqref{eq:bound-cond0}, which can be realized, for instance, via an ISTA step. We further show that FISTA also satisfies a variant of this condition, leading to a quadratic-type convergence behavior of the resulting FISTA–Newton scheme; see Corollary~\ref{cor:quad-poly} below.

Although our {\bf Algorithm~\ref{algo:poly0}} shares some similarities with other Newton methods studied in \cite{Gfrerer25,KhanhMordukhovichPhat25}, the proof of our main result for the quadratic convergence of {\bf Algorithm~\ref{algo:poly0}} is fundamentally different. 

Before proceeding to the main theorem, let us recall the recent result in \cite[Theorem~4.3]{CuiHoheiseiNghiaSun24} stating that $\ox$ is a tilt stable minimizer of $\varphi$ in our setting \eqref{p:CP} if and only if $-\nabla  f(\ox)\in \partial g(\ox)$ and
\begin{equation}\label{eq:TiltPar}
\Ker \nabla^2 f(\ox) \cap \para \partial g^*(-\nabla f(\ox))=\{0\}, 
\end{equation}
which is rather simple to verify at $\ox$ via linear algebra. The appearance of the effective subspace $\para \partial g^*(-\nabla f(\ox))$ in this characterization also serves as another motivation for our {\bf Algorithm~\ref{algo:poly0}} above. 

\begin{Theorem}[Quadratic convergence of the Newton method with effective subspaces] \label{theo:quad-poly}
Consider problem \eqref{p:CP} in which $f:\XX\to \R$ is a twice continuously differentiable and convex function and $g:\XX\to \oR$ is a convex polyhedral  function. Suppose that $\bar{x}$ is a tilt stable minimizer of the function $\varphi$ in \eqref{p:CP}, i.e., $\oz:=-\nabla  f(\ox)\in \partial g(\ox)$ and condition~\eqref{eq:TiltPar} holds, and that the Hessian $\nabla^2 f(x)$ is Lipschitz continuous around $\ox$. Then there exists $\ve>0$ such that for any starting point $x_0\in \mathbb{B}_{\ve}(\ox)$ {and any sequence $\{(y_k,z_k)\}_{k\in \N\cup \{0\}}$ selected during the iterations
that satisfies~\eqref{eq:bound-cond0} with some constant
$\eta>0$ independent of $k$}, we have the assertions: 

\begin{itemize}
    \item[{\bf(i)}] The optimal solution $\{d_k\}$ of problem~\eqref{p:Sub} is well-defined and unique.

    \item[{\bf(ii)}]  The sequence $\{x_k\}$ generated by {\bf Algorithm~\ref{algo:poly0}} converges quadratically to $\ox$.
    
\end{itemize}

\end{Theorem}

\begin{proof}
Since $\ox$ is a tilt stable local minimizer to problem \eqref{p:CP}, the characterization from \eqref{eq:tilt} holds with respect to some constants $\kappa,r>0$. Further, we may shrink $r$ if needed so that there exists $\ell >0$ satisfying
\begin{equation}\label{eq:e-2}
\begin{aligned}
    &\la z,w\ra \ge \kappa \|w\|^2 \quad \forall\, z\in \partial^2 \varphi (x|\, v)(w),\ \forall\, (x,v)\in \gph \partial \varphi\cap \mathbb{B}_{r}(\ox,0),\\
    &\|\nabla f(x)-\nabla f(\ox)\| \le  \ell \|x-\ox\|, \quad \forall\, x\in \mathbb{B}_{r}(\ox),\\
    &\|\nabla^2 f(x)-\nabla^2 f(y)\| \le  \ell \|x-y\|, \quad \forall\, x,y\in \mathbb{B}_{r}(\ox).
\end{aligned}
\end{equation}
Let us pick a universal constant $0<\ve < \min\left\{\dfrac{r}{2\eta\left(1+\ell\right)};\; \dfrac{\kappa}{\ell \eta^2}\right\}$ and set $\overline{C}:=\dfrac{\ell\eta^2}{2\kappa}$. By the choice of $\ve$, it follows that $2\overline{C}\ve < 1$.

We claim {\bf (i)} and {\bf (ii)} by induction as follows. Supposing $x_k\in \mathbb{B}_{\ve}(\ox)$ for some $k\in \N \cup \{0\}$, we will show that the optimal solution $d_k$ to problem \eqref{p:Sub} exists uniquely and the inequalities $\|x_{k+1}-\ox\|\le \overline{C}\|x_k -\ox\|^2 < \ve$ hold. First, observe from step $1$ in {\bf Algorithm~\ref{algo:poly0}} that 
\begin{equation}\label{eq:y0}
\begin{aligned}
    \|y_k - \ox\| \le \eta\|x_k-\ox\| \le \eta \ve < \dfrac{r}{2}.
\end{aligned}
\end{equation}
Define $v_k:=\nabla f(y_k)+z_k\in \partial \varphi(y_k)$. We obtain from \eqref{eq:bound-cond0} that  
\begin{equation}\label{eq:v0}
\|v_k\|=\|\nabla f(y_k)-\nabla f(\ox) +z_k-\oz\|\le \ell\|y_k-\ox\|+\|z_k-\oz\|\le (1+\ell)\eta\ve<\frac{r}{2}.
\end{equation}
One gets from the first inequality of \eqref{eq:e-2} at $(x,v)=(y_k,v_k)\in \gph \partial \varphi\cap \B_r(\ox,0)$ due to \eqref{eq:y0}-\eqref{eq:v0} and the sum rule for the generalized Hessian \cite[Proposition~1.121]{Mordukhovich06} that 
\[
\la \nabla^2 f(y_k)w,w\ra+\la u,w\ra \ge \kk\|w\|^2\quad \mbox{for all}\quad u\in \partial^2 g(y_k|\, z_k)(w). 
\]
By choosing $u=0$, the latter inequality together with the identity from~\eqref{eq:BoTe} gives us that 
\begin{equation}\label{eq:Noz}
\la \nabla^2 f(y_k)w,w\ra\ge \kk\|w\|^2\quad \mbox{for all}\quad w\in L_k=\para\partial g^*(z_k),
\end{equation}
which means $\nabla^2 f(y_k)$ is positive definite over $L_k$. This guarantees the uniqueness of the optimal solution $d_k$ of the quadratic optimization problem~\eqref{p:Sub}.

Next, let us verify the estimates $\|x_{k+1}-\ox\|\le \overline{C}\|x_k-\ox\|^2 < \ve$. To proceed, by using the convexity and polyhedrality of $\epi g$ and the variational geometry of convex polyhedral sets given in \cite[Theorem~2E.3]{Dont-Rock}, we may shrink $r>0$ if needed so that the continuity of $g$ relative to its domain deduces the relation 
\[
T_{\mathrm{epi}\, g}(x,g(x)) \supset T_{\mathrm{epi}\, g}(\ox,g(\ox)) \quad \mbox{for any}\quad x\in \dom g\cap\B_r(\ox)
\]
via the tangent cone in \eqref{eq:tangent}. Then, the polarity between the normal cone and tangent cone to convex sets \cite[Example~6.24]{Rockafellar98} yields 
\begin{equation}\label{eq:normal-subset}
N_{\mathrm{epi}\, g}(x,g(x)) \subset N_{\mathrm{epi}\, g}(\ox,g(\ox)).
\end{equation}
{Now let \(v\in \partial g(x)\). By the characterization \eqref{eq:subdiff-normal-epi}, we have $(v,-1)\in N_{\mathrm{epi} g}(x,g(x))$. Combining this with \eqref{eq:normal-subset}, we obtain $(v,-1)\in N_{\mathrm{epi} g}(\ox,g(\ox)).$
Applying \eqref{eq:subdiff-normal-epi} once again yields \(v\in \partial g(\ox)\). Since \(v\) was chosen arbitrarily, we conclude that
\begin{equation*}
\partial g(x)\subset \partial g(\ox)\quad \text{for all } x\in \dom g\cap \B_r(\ox).
\end{equation*}}
It follows from \eqref{eq:y0} that  
\begin{equation*}
z_k\in \partial g(y_k)\subset \partial g(\ox),
\end{equation*}
which implies that  $\ox,y_k \in \partial g^* (z_k)$. Consequently, we derive that 
\begin{equation}\label{eq:inL_k}
    x_{k+1} - \ox =y_k - d_k-\ox \in \mathrm{span}\, \big(\partial g^* (z_k)-\ox\big)=  L_k.
\end{equation}
By \eqref{eq:Incl}, we may write 
\begin{equation}\label{Fermat-2}
    \nabla^2 f(y_k)d_k = z_k + \nabla f(y_k) - e_k \ \text{ for some }\ e_k \in L_k^{\perp}.    
\end{equation}
Employing \eqref{Fermat-2} gives us that 
\begin{equation}
\begin{aligned}
    \nabla^2 f(y_k)(x_{k+1}-\ox) &= \nabla^2 f(y_k)(y_k-\ox) - \nabla^2 f(y_k)d_k \\
    &= \nabla^2 f(y_k)(y_k-\ox) -z_k - \nabla f(y_k) + e_k \\
    &\, = \nabla^2 f(y_k)(y_k - \ox) + \nabla f(\ox) -\nabla f(y_k) -(z_k-\oz) + e_k.
\end{aligned}\label{eq:Newton-clas-2}
\end{equation}

{We claim next that 
\begin{equation}\label{eq:projzkzbar-2}
P_{L_k}(z_k-\oz)=0.
\end{equation}
Since $\epi g$ is a convex polyhedral set in $\mathbb{X}\times \R$, the Reduction Lemma from \cite[Lemma~2E.4]{Dont-Rock} lends us a neighborhood $U\times V$ of $(0,0)\in (\mathbb{X}\times \R)^2$, independent of $k$, on which the identity 
\begin{equation}\label{eq:reduc-lem}
    (U \times V) \cap \big(\gph N_{\mathrm{epi}\,g} - \big(\ox,g(\ox),\oz,-1\big)\big) = (U \times V) \cap \gph N_{K}
\end{equation}
holds, where $K:=T_{\mathrm{epi}\,g}(\ox,g(\ox))\cap \big\{(\oz,-1)\big\}^{\perp}$ is known as the {\em critical cone} to $\mathrm{epi}\, g$ at $(\ox,g(\ox))$ for $(\oz,-1)$. Let us shrink $r>0$ if necessary so that $\mathbb{B}_r(0)\subset V$, which yields $(z_k-\oz,0)\in V$ by \eqref{eq:bound-cond0}. As $g$ is continuous at $\ox$ relative to $\dom g$ and $U$ is a neighborhood of $(0,0)\in \XX \times \R$, there exists some fixed number $\mu_k>0$ such that 
\begin{equation}\label{eq:ygy-in-U}
\forall\,y\in \partial g^*(z_k)\cap\B_{\mu_k}(\ox)\,\,\Longrightarrow\,\,(y,g(y))\in (\ox,g(\ox))+ U.
\end{equation}
Fix $y\in \partial g^* (z_k)\cap \mathbb{B}_{\mu_k}(\ox)$, we have $z_k\in \partial g(y)$ which yields $(z_k,-1)\in N_{\mathrm{epi}\, g}(y,g(y))$, i.e., 
\begin{equation}\label{eq:gphNepi}
(y,g(y),z_k,-1)\in \gph N_{\mathrm{epi}\, g}.
\end{equation}
It follows from \eqref{eq:ygy-in-U}--\eqref{eq:gphNepi} and $(z_k-\oz,0)\in V$ that  
\[
\big(y-\ox,g(y)-g(\ox),z_k-\oz,0\big)\in (U \times V) \cap \big(\gph N_{\mathrm{epi}\,g} - \big(\ox,g(\ox),\oz,-1\big)\big),
\]
which yields further by \eqref{eq:reduc-lem} the relationship 
\begin{equation*}
    (z_k-\oz,0)\in N_K \big( y-\ox,g(y)-g(\ox)\big) \subset \big(y-\ox,g(y)-g(\ox)\big)^\perp,
\end{equation*}
since $K$ is a closed convex cone. This tells us that
\begin{equation}\label{eq:<>=0}
\la z_k-\oz,y-\ox\ra = \Big\la (z_k-\oz,0),\big(y-\ox,g(y)-g(\ox)\big)\Big\ra =0.
\end{equation}
Hence, we have $z_k-\oz\perp (\partial g^*(z_k)-\ox)\cap \B_{\mu_k}(0)$, which yields  
\[
z_k-\oz\perp \Span \big((\partial g^*(z_k)-\ox)\cap \B_{\mu_k}(0)\big)=\Span \big(\partial g^*(z_k)-\ox\big)=L_k,
\]
where the last equality holds due to the fact that $\ox\in \partial g^* (z_k)$. Our claim~\eqref{eq:projzkzbar-2} is verified.}

From \eqref{eq:Newton-clas-2} and \eqref{eq:projzkzbar-2},  we have
\begin{equation*}
\begin{aligned} P_{L_k}\Big(\nabla^2 f(y_k)(x_{k+1}-\ox)\Big) = P_{L_k}\Big(\nabla^2 f(y_k)(y_k-\ox)+\nabla f(\ox)-\nabla f(y_k)\Big),
\end{aligned}
\end{equation*}
which thus implies by the local Lipschitz continuity of $\nabla^2 f$ around $\ox$ and \eqref{eq:e-2}-\eqref{eq:y0} that
\begin{equation}
\begin{aligned}
    \left\|P_{L_k}\big(\nabla^2 f(y_k)(x_{k+1}-\ox)\big)\right\| 
    &\le \left\|\nabla^2 f(y_k)(y_k-\ox) + \nabla f(\ox)-\nabla f(y_k)\right\| \\
    &=\left\|\int_0^1 \left[\nabla^2 f (y_k)-\nabla^2 f\big(y_k + t(\ox-y_k)\big)\right](y_k-\ox)\,dt\right\|\\
    &\le \int_0^1 \left\|\nabla^2 f (y_k)-\nabla^2 f\big(y_k + t(\ox-y_k)\big)\right\| \|y_k-\ox\|\, dt\\
    &\le \dfrac{\ell}{2}\|y_k-\ox\|^2.
    \label{eq:1-di-2}
\end{aligned}
\end{equation}
In addition, note from \eqref{eq:inL_k} that $x_{k+1} - \ox \in L_k$. Applying \eqref{eq:Noz} with $w=x_{k+1}-\ox\in L_k$ gives us that 
\begin{equation*}
    \la \nabla^2 f(y_k)(x_{k+1}-\ox),x_{k+1}-\ox\ra \ge \kappa \|x_{k+1}-\ox\|^2,
\end{equation*}
which by the fact $x_{k+1}-\ox\in L_k$ and the self-adjoint property of $P_{L_k}$ is equivalent to 
\begin{equation*}
    \Big\la {P_{L_k}\big(\nabla^2 f(y_k)(x_{k+1}-\ox)\big),x_{k+1}-\ox}\Big\ra \ge \kappa \|x_{k+1}-\ox\|^2.
\end{equation*}
The Cauchy-Schwarz's inequality then deduces
\begin{equation}\label{eq:2-di-2}
     \left\|P_{L_k}\big(\nabla^2 f(y_k)(x_{k+1}-\ox)\big)\right\| \ge \kappa \|x_{k+1}-\ox\|.
\end{equation}
Combining \eqref{eq:y0}, \eqref{eq:1-di-2} and \eqref{eq:2-di-2} with the choice of $\ve$ gives us 
\begin{equation*}
\|x_{k+1}-\ox\|\le \dfrac{\ell}{2\kappa}\|y_k-\ox\|^2 \le \dfrac{\ell \eta^2}{2\kappa} \|x_k-\ox\|^2 = \overline{C}\|x_k-\ox\|^2 \le  \dfrac{1}{2}\|x_k-\ox\| < \varepsilon.
\end{equation*}
This justifies our inductive step and thus yields for each $k\in \N \cup \{0\}$ the optimal solution $d_k$ of \eqref{p:Sub} together with the estimate $\|x_{k+1}-\ox\|\le \overline{C}\|x_k-\ox\|^2$, where $\overline{C}$ is independent of $k$. The assertions in {\bf (i)} and {\bf (ii)} are thus justified.
\end{proof}
\vspace{0.05in}

\begin{Remark}{\rm In the proof of our main Theorem~\ref{theo:quad-poly}, the polyhedrality of $g$ plays an essential role, whereas the function $f$ does not need to be convex. Tilt stability at $\ox$ leads to \eqref{eq:Noz}, which guarantees the positive definiteness of $\nabla^2 f(y_k)$ on $L_k$ and the uniqueness of the Newton direction $d_k$ solving \eqref{p:Sub}. However, to be consistent with our later applications as well as the developments for ISTA and FISTA in \cite{BeckTeboulle09,ChambolleDossal15}, which we will use for the approximation step~\eqref{eq:bound-cond0}, and to avoid additional technical details, we still keep the assumption that $f$ is convex throughout the paper.     
}    
\end{Remark}

As discussed after {\bf Algorithm~\ref{algo:NewCod}}, a practical way to choose $(y_k,z_k)\in \gph \partial g$ satisfying \eqref{eq:bound-cond0} is using the proximal gradient (a.k.a. ISTA) step as in \eqref{eq:prox}. We incorporate this step into our {\bf Algorithm~\ref{algo:poly0}} as follows and prove the corresponding quadratic convergence result for the sake of completeness.
\begin{algorithm}[H]
\caption{ISTA-Newton method with effective subspaces}
\textbf{Input}: $x_0\in \mathbb{X}, \alpha>0$. 
\begin{algorithmic}[1]\label{algo:poly-ista}
  \STATE Compute $y_k := \prox_{\alpha g}(x_k-\alpha\nabla f(x_k))$.
  \STATE Compute $z_k := \frac{x_k-y_k}{\alpha} - \nabla f(x_k)$.
  \STATE Compute $L_k:={\rm par}\, \partial g^*(z_k)$ and find an optimal solution $d_k$ of
  \begin{equation*}
  \min_{d\in L_k}\quad \frac{1}{2}\la \nabla^2 f(y_k)d,d\ra-\la z_k+\nabla f(y_k),d\ra.
  \end{equation*}
  \STATE Update $x_{k+1}=y_k-d_k$. 
\end{algorithmic}
\end{algorithm}

\begin{Corollary}[Quadratic convergence of the ISTA-Newton method with effective subspaces] \label{coro:PGN} 
Under the assumptions of Theorem \ref{theo:quad-poly}, there exists $\varepsilon>0$ such that for each $x_0 \in \mathbb{B}_{\varepsilon}(\ox)$, {\bf Algorithm~\ref{algo:poly-ista}} produces a sequence $\{x_k\}$ converging quadratically to $\ox$.
\end{Corollary}
\begin{proof}
The algorithm starts with some parameter $\alpha>0$. To ensure the quadratic convergence via the proof by induction in Theorem \ref{theo:quad-poly}, it suffices to show the fulfillment of our boundedness condition \eqref{eq:bound-cond0} at each pair $(y_k,z_k)\in \gph \partial g$ $(k=0,1,\ldots)$, given that $x_k\in \mathbb{B}_{\ve}(\ox)$. Indeed, {\bf Algorithm~\ref{algo:poly-ista}} forces $z_k \in \partial g(y_k)$ due to \eqref{eq:res}. Let us set 
\begin{equation}\label{eq:eta}
\eta:=\dfrac{(2+\alpha)(1+\ell \alpha)}{\alpha},
\end{equation}
where $\ell$ is a Lipschitz modulus of $\nabla f$ around $\ox$. Pick $x_k\in \mathbb{B}_{\varepsilon}(\ox)$, where $\varepsilon$ is chosen as in Theorem \ref{theo:quad-poly}, satisfying in addition $\ve <r$ from \eqref{eq:e-2}. The {\em nonexpansiveness} property of $\prox_{\alpha g}$ implies that
\begin{equation}\label{eq:y0xbar-rem-1}
\begin{aligned}
    \|y_k - \ox\| &= \left\|\prox_{\alpha g} \big(x_k - \alpha \nabla f(x_k)\big) - \prox_{\alpha g} \big(\ox - \alpha \nabla f(\ox)\big)\right\| \\
    &\le \big\|(x_k -\ox) - \alpha \big(\nabla f(x_k) - \nabla f(\ox)\big)\big\| \le \big(1+\ell {\alpha}\big)\|x_k - \ox\|.
\end{aligned}
\end{equation}
Moreover, it follows from \eqref{eq:y0xbar-rem-1} and the Lipschitz continuity of $\nabla f$ on $\mathbb{B}_{\varepsilon}(\ox)$ that
\begin{equation*}
\begin{aligned}
    \|z_k - \oz\| &= \left\|\dfrac{x_k - y_k}{\alpha} - \nabla f(x_k) + \nabla f(\ox)\right\| \le \dfrac{\|x_k - \ox\|}{\alpha}+\dfrac{\|y_k - \ox\|}{\alpha} + \ell \|x_k - \ox\|,
\end{aligned}
\end{equation*} 
which together with \eqref{eq:y0xbar-rem-1} and \eqref{eq:eta} verifies \eqref{eq:bound-cond0} at the pair $(y_k,z_k)$. Our conclusion then follows from the inductive step as in the proof of Theorem~\ref{theo:quad-poly}.
\end{proof}

\begin{Remark}\label{rem:compare-manifold}
\rm (Comparisons with active submanifolds Newton methods in \cite{BIM23,LewisWylie2021,wylie2019}) 
{This remark clarifies the relation between our effective subspaces and those in \cite{BIM23,LewisWylie2021,wylie2019}, highlighting both the cases where the identification procedures coincide and where they differ.}

{{\bf Differences:}} Compared to the Newton method with active Riemannian submanifolds built in \cite[Algorithms~1~\&~2]{BIM23} which also achieves quadratic convergence, our convergence theory does not require the \emph{non-degeneracy condition} 
\begin{equation}\label{eq:non-de}
    -\nabla f(\ox)\in \mathrm{ri}\, \partial g(\ox)
\end{equation}
in \cite[Definition~3.2]{BIM23}. {This condition can be interpreted as a generic property in the polyhedral setting. Indeed, for $\ox\in \dom \partial g$, the set $\partial g(\ox)$ is a polyhedron. Its relative interior is the unique face of maximal dimension in $\aff \partial g(\ox)$, while its boundary is a finite union of lower-dimensional faces. Consequently, a point chosen at random from $\partial g(\ox)$ belongs to $\ri \partial g(\ox)$ with probability one (with respect to the Lebesgue measure on $\aff \partial g(\ox)$).} The non-degeneracy property is indispensable in the construction of {\em r-structure critical point} \cite{BIM23}, is used to determine a manifold $\mathcal{M}$ on which $g$ is $\mathcal{C}^2$ and that the iterates eventually lie on this manifold. However, this condition is not always true, as shown in the following example.

On $\mathbb{X}=\R^2$, consider the functions 
\begin{equation*}
    f(x_1,x_2):= \dfrac{1}{2}\big\|(x_1,x_2)-(2,-1)\big\|^2 \ \text{ and }\ g(x_1,x_2):= |x_1|+|x_2|,\quad \text{ for }\quad (x_1,x_2)\in \R^2.
\end{equation*}
Note that $\nabla f(1,0) = (-1,1)$ and $\partial g(1,0) = \{1\} \times [-1,1]$. Hence, $\ox:= (1,0)$ is a local minimizer of $f+g$, as it satisfies the first-order optimality condition $0 \in \nabla f(\ox) + \partial g(\ox)$. Moreover, since $f+g$ is strongly convex, $\ox$ is a tilt stable minimizer to $f+g$. By performing as in {\bf Algorithm~\ref{algo:poly0}} with a starting point close enough to $\ox$, our iterates converge quadratically to $\ox$, as a result of Theorem \ref{theo:quad-poly}. However, observe that 
\begin{equation*}
    -\nabla f(\ox) = (1,-1)\notin \{1\}\times (-1,1) =  \mathrm{ri}\, \partial g(\ox),
\end{equation*}
i.e., the non-degeneracy condition is violated at $\ox$. 

Note further that in the setting of \cite[Example~2.3]{BIM23}, their \emph{optimal submanifold} for this example is $\mathcal{M}=\R \times \{0\}$. {At first sight}, it may be natural to think this submanifold is the effective subspace $L:=\para \partial g^*(-\nabla f(\ox))$ in our approach. However, they are not the same, as  
\[
\para \partial g^*(-\nabla f(\ox))=\para \big(N_{\B_\infty}(1,-1)\big)=\para (\R_+\times \R_-)=\R^2,
\]
where $\B_\infty$ is the $\ell_\infty$ norm unit ball in $\R^2$. Moreover, the $\ell_1$ norm $g(x)$ is not smooth on $L$. This somewhat distinguishes our algorithms from those in \cite{BIM23,LewisWylie2021,wylie2019}. More specifically, for solving the more general optimization problem \eqref{p:CP} with $g(x)=\|x\|_1$ in $\R^n$, the update of the active submanifold/subspace at iteration $k$ in \cite[Example~2.3]{BIM23} and \cite[Page 97]{wylie2019} only relies on $y_k$ as follows
\[
\mathcal{M}(y_k) = T_{y_k}\mathcal{M}(y_k)=\{x\in \R^n\mid  x_i=0 \mbox{ for } i\in I_k\}\quad \mbox{with}\quad I_k:=\{i\in \{1,\ldots,n\}\mid  (y_k)_i=0\}. 
\]
{This is usually referred to as} the {\em active submanifold/subspace} as it is built up based on the {\em active set} at $y_k$. 
However, our effective subspace $L_k$ used to update Newton direction $d_k$ in \eqref{eq:NewStep} relies on $z_k\in \partial g(y_k)$ as  
\begin{equation}\label{eq:Lk-l1}
L_k=\para \big(N_{\B_\infty}(z_k)\big)=\{x\in \R^n\mid  x_j=0 \mbox{ for } j\in J_k\}\quad  \mbox{with}\quad J_k:=\{j\in \{1,\ldots,n\} \mid  |(z_k)_j|<1\}; 
\end{equation}
see also formula \eqref{eq:spanN} in Section~\ref{sec:exam}. 
Note that $J_k\subset I_k$ and therefore our effective subspaces $L_k$ are larger than $\mathcal{M}(y_k)$. To make this distinction precise, Example~\ref{ex:vary-subspace} presents an instance of problem~\eqref{p:CP} together with a sequence $(y_k,z_k)$ satisfying Step~1 of {\bf Algorithm~\ref{algo:poly0}} for which the effective subspaces $L_k$ do not stabilize, whereas the corresponding active subspaces $T_{y_k}\mathcal{M}(y_k)$ remain unchanged. Thus, the subspaces $L_k$ can vary even when the active manifold structure is fixed, and may extend beyond regions where $g$ admits a smooth local representation.
In the same spirit, the $\ell_1$ norm may not be smooth on $L_k$ in \eqref{eq:Lk-l1}. Despite these differences, it seems that the performance of {\bf Algorithm~\ref{algo:poly0}} differs only slightly between these two updates when solving the $\ell_1$ regularized optimization problems. The main advantage of our construction is that it allows us to prove the quadratic convergence of {\bf Algorithm~\ref{algo:poly0}} without the non-degeneracy condition in \cite{BIM23,LewisWylie2021,wylie2019}. More importantly, a closed-form computation of our effective subspaces $L_k$ can be obtained explicitly by employing only the first-order information from $g^*$ (or from $g$, via \eqref{eq:Lk-via-g}) for any polyhedral function $g$; see also Section~\ref{sec:exam} for explicit computations of other polyhedral regularizers.

{From a practical viewpoint, removing the non-degeneracy condition is not merely of theoretical interest. When $z_k$ lies on the relative boundary of $\partial g(y_k)$, the identified subspaces may differ from those obtained via active manifold identification, which in turn affects the resulting Newton direction. Such boundary effects may influence the numerical behavior; see, e.g., \cite[Subsection~3.2]{FMP18}.}
\\

{{\bf Similarities:} Consider the primal and dual stratifications from \cite[Subsubsection~2.3.5]{FMP18}: $$\mathscr{M}:=\big\{\mathcal{M}(x)\mid x\in \dom \partial g\big\} \quad \text{and}\quad \mathscr{M}^*:=\big\{\mathcal{M}(x^*)\mid x^*\in \dom \partial g^*\big\}.$$ Under the {\em non-degeneracy} condition \eqref{eq:non-de}, together with the graph convergence $(y_k,z_k)\xrightarrow{{\rm gph}\, \partial g} (\ox,\oz)$ as $k\to \infty$, the sequence $(y_k,z_k)$ generated by {\bf Algorithm~\ref{algo:poly0}} satisfies 
\begin{equation}\label{eq:zk-ri-yk}
    z_k \in \ri \partial g (y_k) \quad \text{for}\, \, k\in \N \text{ large}. 
\end{equation}
This implies the identification
\begin{equation}\label{eq:Lk=Mk}
    L_k = T_{y_k}\mathcal{M} (y_k),
\end{equation}
where $T_{y_k}\mathcal{M} (y_k)$ denotes the tangent cone at $y_k$ to the primal stratum $\mathcal{M}(y_k)$. As a consequence, the Newton steps in {\bf Algorithm~1} align with those in \cite[Algorithms~1~and~2]{BIM23}.}
\\
{{\em Proof of \eqref{eq:zk-ri-yk}--\eqref{eq:Lk=Mk}.} As noted in \cite[page~2997]{FMP18}, the set $\mathrm{ri}\,\partial g(\cdot)$ stays the same over a primal stratum. Consequently, 
\begin{equation}\label{eq:ri-constant}
    \mathcal{J}_g \big(\mathcal{M}(y_k)\big) := \bigcup_{x'\in \mathcal{M}(y_k)}\ri \partial g(x') =  \ri \partial g(y_k).
\end{equation}
By \eqref{eq:non-de} and \cite[Proposition~2]{FMP18} we have $$\oz\in \ri \partial g(\ox)\subset \mathcal{J}_g \big(\mathcal{M}({\ox})\big) \in \mathscr{M}^*,$$ 
and hence $\mathcal{M}^* (\oz)=\mathcal{J}_g \big(\mathcal{M}({\ox})\big)$. We may now employ \cite[Theorem~1~and~Definition~4]{FMP18} and the convergence $(y_k,z_k)\xrightarrow{{\rm gph}\, \partial g}(\ox,\oz)$ to get the relations
\begin{equation*}
    \mathcal{M}(\ox) \subset \mathcal{M} (y_k) \subset \mathcal{J}_{g^*} \big(\mathcal{M}^*(\oz)\big)=\mathcal{J}_{g^*} \big(\mathcal{J}_g \big(\mathcal{M}({\ox})\big)\big) = \mathcal{M}({\ox}) \qquad \text{ for}\, \, k\in \N \text{ large},
\end{equation*}
and thus $\mathcal{M} (y_k) = \mathcal{M}({\ox})$ for $k\in \N$ large enough. Again by the same theorem, applied dually,
\begin{equation*}
    \mathcal{M}^*(z_k) =\mathcal{M}^*({\oz})\qquad \text{ for}\, \, k\in \N \text{ large}.
\end{equation*}
Therefore, \eqref{eq:ri-constant} implies that
\begin{equation*}
    z_k \in \mathcal{M}^*(z_k) = \mathcal{M}^*({\oz}) = \mathcal{J}_g \big(\mathcal{M}({\ox})\big) = \mathcal{J}_g \big(\mathcal{M}(y_k)\big) = \ri \partial g (y_k).
\end{equation*}
The inclusion in \eqref{eq:zk-ri-yk} is justified.}

{Now let us verify \eqref{eq:Lk=Mk}. Since $z_k\in \mathcal{M}^* (z_k)$, the proof of \cite[Proposition~2]{FMP18} implies that
\begin{equation*}
    \ri \partial g^* (z_k) =\mathcal{J}_{g^*}\big(\mathcal{M}^*(z_k)\big) = \mathcal{J}_{g^*}\big(\mathcal{M}^*(\oz)\big) = \mathcal{M}(\ox)=\mathcal{M}(y_k).
\end{equation*}
Since $\ri \partial g^* (z_k)$ is nonempty and $\mathcal{M}(y_k)$ is an affine manifold, we have
\begin{equation*}
    L_k := \para \partial g^* (z_k) = \para \big(\ri \partial g^*(z_k)\big) = \para \mathcal{M}(y_k) = T_{y_k} \mathcal{M}(y_k).
\end{equation*}
The proof is complete.}

{In fact, from the proof above, a sharper identification of \eqref{eq:Lk=Mk} is 
\begin{equation*}
L_k = T_{y_k}\mathcal{M}(y_k) = T_{\ox}\mathcal{M}({\ox}) \qquad \text{for all sufficiently large } k.
\end{equation*}
In particular, the effective subspaces $\{L_k\}$ and active subspaces $\{T_{y_k}\mathcal{M}(y_k)\}$ eventually stabilize, under the non-degeneracy condition \eqref{eq:non-de}. As we do not know the solution $\ox$, using  effective subspaces $L_k$ or active subspaces $T_{y_k}\mathcal{M}(y_k)$ instead of $T_{\ox}\mathcal{M}({\ox})$ is more suitable in these algorithms.}

{To summarize, although the non-degeneracy condition is generically satisfied, handling boundary cases is essential for robustness and may significantly influence the numerical behavior of the method.}
\endproof

{In the next example, we demonstrate concretely the distinction between the effective subspaces generated by our method and the active manifolds/subspaces studied in \cite{BIM23,LewisWylie2021,wylie2019} without the non-degeneracy condition. Specifically, we exhibit a case where the active manifolds/subspaces remain fixed along the iterates, and the non-degeneracy condition fails at the optimal solution $\bar x$, whereas the effective subspaces $L_k$ in {\bf Algorithm~\ref{algo:poly0}} nevertheless vary from one iteration to the next.}

\begin{Example}[Varying effective subspaces versus a fixed active manifold] \label{ex:vary-subspace}
\rm
{
On $\XX=\R^2$, consider the optimization problem
\[
\min_{x \in X} \big[f(x) + g(x)\big],
\]
where the functions $f(x_1,x_2)$ and $g(x_1,x_2)$ are defined respectively as
\[
f(x_1,x_2):=\frac12(x_2+1)^2+\frac12(x_1-2)^2+(x_1-1)^4
\quad\text{and}\quad
g(x_1,x_2):=|x_1|+|x_2|.
\]
By direct computation, one can easily derive that
\[
\nabla f(x_1,x_2)=\big(x_1-2+4(x_1-1)^3,\ x_2+1\big)^T \quad \text{and} \quad \nabla^2 f(x_1,x_2)=
\begin{pmatrix}
1+12(x_1-1)^2 & 0\\
0 & 1
\end{pmatrix}.
\]
Observe that $\nabla f(1,0)=(-1,1)^T$ and $\partial g(1,0)=\{1\}\times[-1,1]$. Hence, $\bar x:=(1,0)$ is a (local) minimizer of $f+g$, since it satisfies the Fermat's first-order optimality condition. Moreover, since $f+g$ is strongly convex, $\bar x$ is a tilt-stable minimizer of $f+g$.}

{We next construct a sequence $(y_k,z_k)$ satisfying Step~1 of {\bf Algorithm~\ref{algo:poly0}} for which the effective subspaces $L_k$ vary with $k$, while the active manifolds used in \cite{BIM23,LewisWylie2021,wylie2019} remain unchanged. Let $\{a_k\}_{k=-1}^{\infty}\subset (0,1]$ be any sequence satisfying $a_k\downarrow 0$ and $a_k \le \dfrac{4a_{k-1}^3}{1+12a_{k-1}^2}$ for each $k\ge 0$. For $k=0,1,\ldots$, define
\begin{equation}\label{eq:xkykak}
x_k:=\left(1+\dfrac{8a^3_{k-1}}{1+12a^2_{k-1}},0\right) \quad \text{and} \quad y_k:=(1+a_k,0),
\end{equation}
together with
\[
z_k:=
\begin{cases}
(1,-1)^T, & \text{if }k\text{ is even},\\[1mm]
\left(1,-1+\dfrac{a_k^2}{2}\right)^T, & \text{if }k\text{ is odd}.
\end{cases}
\]
Since $y_k=(1+a_k,0)$ with $1+a_k>0$, we have $\partial g(y_k)=\{1\}\times[-1,1].$ Therefore, $z_k\in \partial g(y_k)$ for every $k$, and thus $(y_k,z_k)\in \gph\partial g$. Set $\bar z:=-\nabla f(\bar x)=(1,-1)^T$. Then $\|y_k-\bar x\|=a_k$, $\|x_k-\ox\|=\dfrac{8a^3_{k-1}}{1+12a^2_{k-1}}$ and
\[
\|z_k-\bar z\|=
\begin{cases}
0, & \text{if }k\text{ is even},\\[1mm]
\dfrac{a_k^2}{2}, & \text{if }k\text{ is odd}.
\end{cases}
\]
Hence, we deduce that
\[
\|z_k-\bar z\|+\|y_k-\bar x\|
\le \frac{a_k^2}{2} + a_k
\le 2a_k
\le \dfrac{8a^3_{k-1}}{1+12a^2_{k-1}} = \|x_k-\ox\|,
\]
which shows that the boundedness condition \eqref{eq:bound-cond0} is fulfilled with $\eta=1$.}

{Let us now compare the active manifolds from \cite{BIM23,LewisWylie2021,wylie2019} and our effective subspaces. Since $y_k=(1+a_k,0),$ their active index sets are
\[
I_k:=\{i\in\{1,2\}\mid (y_k)_i=0\}=\{2\}.
\]
Hence, their active manifolds are
\[
\mathcal{M}(y_k) :=\{(u_1,u_2)\in\R^2\mid u_2=0\}=\R\times\{0\}.
\]
In particular, the active subspace is $T_{y_k} \mathcal{M}(y_k)=\R\times\{0\}$ and does not change with $k.$}

{On the other hand, using the formula \eqref{eq:spanN} for the $\ell_1$ norm, the effective subspace is
\[
L_k=\para\partial g^*(z_k)=\para N_{\mathbb{B}_\infty}(z_k)
=\{u\in\R^2\mid u_j=0\ \text{whenever } |(z_k)_j|<1\}.
\]
If $k$ is even, then $z_k=(1,-1)^T$ and both coordinates lie on the boundary of $\mathbb{B}_\infty$, we have $L_k=\R^2.$ If $k$ is odd, then the second coordinate of $z_k$ satisfies $\big|-1+\frac{a_k^2}{2}\big|<1$, which implies that $L_k=\R\times\{0\}$. Therefore, the effective subspaces generated by {\bf Algorithm~\ref{algo:poly0}} alternate as
\[
L_k=
\begin{cases}
\R^2, & \text{if } k \text{ is even},\\[1mm]
\R\times\{0\}, & \text{if } k \text{ is odd}.
\end{cases}
\]
while the active subspace $T_{y_k}\mathcal{M}(y_k)=\R\times\{0\}$ remains fixed for every $k$.}

{Observe also that $\ri \partial g(\bar x)=\{1\}\times(-1,1),$ which yields $-\nabla f(\bar x)=(1,-1)^T\notin\ri\partial g(\bar x).$
Hence, the non-degeneracy condition fails in this example.}

{Finally, let us justify that $\{x_k\}$ in \eqref{eq:xkykak} is indeed the iterates produced by our nonsmooth Newton method in {\bf Algorithm~\ref{algo:poly0}}. First, we have
\[
\nabla^2 f(y_k)=
\begin{pmatrix}
1+12a_k^2 & 0\\
0 & 1
\end{pmatrix}
\quad \text{and}\quad
\nabla f(y_k)+z_k=
\begin{cases}
(a_k+4a_k^3,0)^T, & \text{if }k\text{ is even},\\[1mm]
\left(a_k+4a_k^3,\dfrac{a_k^2}{2}\right)^T, & \text{if }k\text{ is odd}.
\end{cases}
\]
Hence, the Newton subproblem \eqref{p:Sub} has the same unique solution in both cases, namely
\[
d_k=
\left(
\frac{a_k+4a_k^3}{1+12a_k^2},\,0
\right).
\]
Therefore, the update step is 
\[
y_k-d_k
=
\left(
1+a_k-\frac{a_k+4a_k^3}{1+12a_k^2},\,0
\right)
=
\left(
1+\frac{8a_k^3}{1+12a_k^2},\,0
\right) = x_{k+1}.
\]
So $\{x_k\}$ is precisely the iterates generated by {\bf Algorithm~\ref{algo:poly0}}.}

{In conclusion, the sequence $\{x_k\}$ converges to $\ox$ as $k\to \infty$ but the effective subspaces $L_k$ alternate after each step.}
\end{Example}


\end{Remark}

Next, we establish several connections between our strategy of determining Newton directions with other nonsmooth Newton schemes developed recently in \cite{GfrererCOA25,Gfrerer25,GfrererSIOPT21,Gfrerer22,KhanhMordukhovichPhat25,KhanhMordukhovichPhat24,MordukhovichSarabi21} for the special case when $g(x)$ is either the indicator or support function to a polyhedral set. 

\begin{Proposition}\label{prop:compare}
\rm 
Let $C$ be a polyhedral convex set in $\XX$. Suppose that  $g$ is either the indicator function $\delta_C$ or the support function $\sigma_C$. With $(\ox,\oz)\in \gph \partial g$ and $L=\para\partial g^*(\oz)$,  
there holds 
\begin{equation}\label{eq:LinN}
    L^\perp \times L \subset N_{\mathrm{gph}\,\partial g}(\ox,\oz).
\end{equation}
Moreover, we also have 
\begin{equation}\label{eq:SCD}
    L \times L^\perp \in \mathcal{S}\partial g(\ox,\oz),
\end{equation}
where $\mathcal{S}\partial g$ is the {\em SC derivative} \cite[Definition~3.3]{Gfrerer22} of the subdifferential mapping $\partial g$.
\end{Proposition}
\begin{proof}
{\bf Case I}: $g(x)=\delta_C(x)$, the indicator function of $C$. Denote by
\begin{equation}\label{eq:Crit}
K:=K_C (\ox,\oz)=T_C(\ox)\cap\{\oz\}^\perp
\end{equation}
the critical cone of $C$ at $\ox$ for $\oz$. Recall that  
a cone $F$ in $\XX$ is called a {\em closed face} of $K$ if there exists some $v$ in $K^*$, the polar cone of $K$, such that $F=K\cap\{v\}^\perp$. The limiting normal cone to $\gph \partial g$ at $(\ox,\oz)$ is expressed by  \cite[Proposition~4.4]{PoliquinRock98} 
\begin{equation}\label{eq:Normal}
   N_{{\rm gph}\, \partial g}(\ox,\oz)=\{(F_1-F_2)^*\times (F_1-F_2) \mid \mbox{$F_1, F_2$ are closed faces of $K$ with } F_2\subset F_1\}.
\end{equation}
As $L=\{w\mid (0,w)\in N_{{\rm gph}\, \partial g}(\ox,\oz)\}$ by \eqref{eq:BoTe} and $0\in (F_1-F_2)^*$ for any choice of closed faces $F_1, F_2$ of $K$ in the above formula, we have
\[
L=\bigcup\{F_1-F_2\mid \mbox{$F_1, F_2$ are closed faces of $K$ with } F_2\subset F_1\}. 
\]
As $F_1-F_2\subset K-K$  for any pair of closed faces $F_1, F_2$ of $K$, and a particular choice is $F_1=F_2=K$, the above identity gives us $L=K-K.$ This together with~\eqref{eq:Normal} verifies \eqref{eq:LinN}. 

To justify \eqref{eq:SCD}, let us recall the formula in \cite[Equation~(29) in Example~3.29]{Gfrerer22} that 
\begin{equation}\label{eq:SCD0}
\mathcal{S}\partial g(\ox,\oz)=\mathcal{S}N_C(\ox,\oz)=\{(F-F)\times (F-F)^\perp\,|\, F \mbox{ is a closed face of } K\}. 
\end{equation}
As $L=K-K$, we obtain \eqref{eq:SCD} immediately  by choosing $F=K$ from the above formula.
 
{\bf Case II.} $g(x)=\sigma_C(x)=\delta_C^*(x)$. In this case, $L=\para \partial \delta_C(\oz)=\para N_C(\oz)$. Since $\partial g^*=(\partial g)^{-1}$, to justify \eqref{eq:LinN}, we just need to show that 
\begin{equation}\label{eq:SCD1}
L \times L^\perp\subset N_{{\rm gph}\, \partial g^*}(\oz,\ox)=N_{{\rm gph}\,N_C}(\oz,\ox). 
\end{equation}
The critical cone in this case has to be adjusted to $H=T_C(\oz)\cap\{\ox\}^\perp$. We claim next that $(\para N_C(\oz))^\perp$ is a closed face of $H$. Suppose that $C$ can be represented by
\[
C=\{z\in \XX\mid \la a_i,z\ra\le b_i,\, i=1,\ldots,n\}.
\]
With $\oz\in C$, we have 
\[
N_C(\oz)=\cone\{a_i\mid i\in I(\oz)\}\quad \mbox{with}\quad I(\oz):=\{i\in \{1, \ldots,n\}\mid \la a_i,\oz\ra=b_i\}. 
\]
Obviously, $\Span N_C(\oz)=\Span \{a_i\mid i\in I(\oz)\}$. Moreover, the tangent cone to $C$ at $\oz$ is computed by
\[
T_C(\oz)=\{d\in \XX\mid \la a_i,d\ra\le 0,\, i\in I(\oz)\}. 
\]
As $\ox\in N_C(\oz)$, we may write 
\[
\ox=\sum_{i\in J}\lm_i a_i
\]
for some index set $J\subset I(\oz)$ with $\lm_i>0$, $i\in J$. Thus, the critical cone is obtained by 
\[
H=T_C(\oz)\cap\{\ox\}^\perp=\{d\in \XX\mid \la a_i,d\ra=0,\, i\in J, \la a_i,d\ra\le 0,\, i\in I(\oz)\setminus J\}. 
\]
Define $u:=\sum_{i\in I(\oz)\setminus J} a_i$ if $ J\neq I(\oz)$ and $u:=0$ if $J= I(\oz)$. Note that $u\in H^*$ and that the closed face $F_1:=H\cap \{u\}^\perp$ is computed by 
\[
\begin{aligned}
F_1&=\{d\in \XX\mid \la a_i,d\ra=0,\, i\in J, \la a_i,d\ra\le 0,\, i\in I(\oz)\setminus J,  \sum_{i\in I(\oz)\setminus J} \la a_i,d\ra=0\}\\
&=\{d\in \XX\mid \la a_i,d\ra=0,\, i\in J, \la a_i,d\ra=0,\, i\in I(\oz)\setminus J\},
\end{aligned}
\]
which is exactly $(\para N_C(\oz))^\perp$. Combining this with the formula \eqref{eq:Normal} with $F_2=0$ (but replacing $(\ox,\oz)$ there with $(\oz,\ox)$ and $K$ with $H$) gives us that 
\[
L\times L^\perp= (\para N_C(\oz))^{\perp\perp}\times (\para N_C(\oz))^{\perp}\subset N_{{\rm gph}\, \partial g^*}(\oz,\ox),
\]
which is exactly \eqref{eq:SCD1}. Hence, we also have \eqref{eq:LinN} by interchanging the roles of $L$ and $L^\perp$ from $N_{{\rm gph}\, \partial g^*}(\oz,\ox)$ to $N_{{\rm gph}\, \partial g}(\ox,\oz)$. The inclusion \eqref{eq:SCD} is derived similarly by replacing formula \eqref{eq:Normal} with \eqref{eq:SCD0}. 
\end{proof}

\begin{Corollary}\label{coro:compare} Consider problem \eqref{p:CP} in which $g$ is either the indicator function or the support function of a polyhedral set in $\XX$. {In this case,} 
{\bf Algorithm~\ref{algo:poly-ista}} is a special case of {\bf Algorithm~\ref{algo:NewCod}} and {\rm \cite[Algorithm~4.2]{Gfrerer25}}.
\end{Corollary}
\begin{proof}
Proceeding as in {\bf Algorithm~\ref{algo:poly-ista}}, at each iteration, $d_k$ solves the following linear system:  
\begin{equation*}
    -\nabla^2 f(y_k) d_k + z_k + \nabla f(y_k) \in L_k^\perp \quad \text{and} \quad d_k \in L_k.
\end{equation*}
On the one hand, by using \eqref{eq:LinN}, we get 
\begin{equation*}
    \Big( -\nabla^2 f(y_k) d_k + z_k + \nabla f(y_k),-d_k\Big) \in L_k^\perp \times L_k \subset N_{{\rm gph}\, \partial g}(y_k,z_k),
\end{equation*}
which thus implies by \eqref{def:limit-code} and \eqref{def:secondsub} that $
    -\nabla^2 f(y_k)d_k + z_k + \nabla f(y_k)  \in \partial^2 g (y_k|\,z_k)(d_k).$
This shows that our Newton step $d_k$ satisfies the coderivative inclusion \eqref{eq:NewStep} from {\bf Algorithm~\ref{algo:NewCod}}.

On the other hand, employing \eqref{eq:SCD} also indicates that our $-d_k$ can be chosen as a Newton direction in \cite[Algorithm~4.2]{Gfrerer25}. 
\end{proof}

\begin{Remark}\rm 
Corollary \ref{coro:compare} shows that when $g$ is either the indicator or the support function of a convex polyhedral set in $\XX$, our Newton scheme provides an efficient way of choosing Newton directions that are special instances of the coderivative-based \cite{KhanhMordukhovichPhat25} and the SCD Newton method \cite[Algorithm~4.2]{Gfrerer25} while avoiding computing any second-order structures of the regularizer $g$. In contrast to these approaches, which, under their standing assumptions, guarantee at most local superlinear convergence, {our method uses only first-order information from $g$, as demonstrated in \eqref{eq:Lk-via-g},} yet still attains quadratic convergence. {We emphasize, however, that identifying the full effective subspace may itself require a nontrivial analysis of the subdifferential geometry of $g$. At present, both this construction and our quadratic-convergence analysis are developed specifically for convex piecewise linear/polyhedral functions $g$, and extending them beyond this setting is a research direction currently in progress.}\endproof
\end{Remark} 
{It is worth noting that}, in practice, the proximal gradient (ISTA) step may produce only small improvements in the early iterations. To make the algorithm faster, an accelerated version known as Fast Proximal Gradient Method (a.k.a. {\em Fast Iterative Shrinkage-Thresholding Algorithm} (FISTA)) \cite{BeckTeboulle09,ChambolleDossal15} is usually applied. In practice, since we do not know how to choose $x_0$ close to $\ox$ as required in Theorem~\ref{theo:quad-poly}, using FISTA may move the iterates quickly to a neighborhood of $\ox$ and then allow us to start the Newton step, which exhibits quadratic convergence. We modify {\bf Algorithm~\ref{algo:poly-ista}} using FISTA below.

\begin{algorithm}[H]
\caption{FISTA-Newton method with effective subspaces}
\textbf{Input}: $x_0 = x_{-1}\in \mathbb{X}, \alpha >0, \{\beta_k\} \subset (0,1]$.
\begin{algorithmic}[1]\label{algo:poly-fista}
  \STATE Compute $u_k := x_k + \beta_k(x_k-x_{k-1})$.
  \STATE Compute $y_k := \prox_{\alpha g}(u_k-\alpha\nabla f(u_k))$.
  \STATE Compute $z_k := \frac{u_k-y_k}{\alpha} - \nabla f(u_k)$.
  \STATE Compute $L_k:={\rm par}\, \partial g^*(z_k)$ and find $d_k$ solving problem~\eqref{p:Sub}.
  \STATE Update $x_{k+1}=y_k-d_k$. 
\end{algorithmic}
\end{algorithm}

Here, $\{\beta_k\}$ is a bounded below sequence in $(0,1]$; see \cite{BeckTeboulle09,ChambolleDossal15,LiangLuoTao22} for different ways of updating this sequence. We remark that the practice of combining FISTA with a second-order method for the $\ell_1$ logistic regression problem appeared in \cite[Page~99-100]{wylie2019}, where promising numerical performance was reported. However, the convergence analysis of this combined method is not provided there. In the following corollary, we establish a definitive two-step quadratic-type convergence result for FISTA when employed within our effective subspace framework.

Unlike the scheme with ISTA-Newton update for $(y_k,z_k)$ that satisfies \eqref{eq:bound-cond0} immediately, it is not clear to us if the choice of $(y_k,z_k)$ in {\bf Algorithm~\ref{algo:poly-fista}} also fulfills this condition because of the extrapolation step above, which involves both $x_k$ and $x_{k-1}$. However, it is similar to \eqref{eq:bound-cond0} proved in Corollary~\ref{coro:PGN} that 
\begin{equation}\label{eq:fis-bound}
\begin{aligned}
\|y_k-\bar x\|+\|z_k-\bar z\|&\le \eta\|u_k-\ox\|\\
&=\eta\|(1+\beta_k)(x_k-\ox)-\beta_k(x_{k-1}-\ox)\|\\
&\le {3}\eta\max\{\|x_k-\ox\|;\,\|x_{k-1}-\ox\|\}
\end{aligned}
\end{equation}
with $\eta>0$ in \eqref{eq:eta}. {By adapting the proof of Theorem~\ref{theo:quad-poly} to incorporate the bound in \eqref{eq:fis-bound}, we establish in Corollary~\ref{cor:quad-poly} a two-step quadratic-type convergence result for the auxiliary sequence
\begin{equation}\label{eq:pk}
p_k:=\max\{\|x_k-\ox\|,\|x_{k-1}-\ox\|\}.
\end{equation}
More precisely, when FISTA is employed within our effective subspace framework, the sequence \(\{p_k\}\) satisfies $p_{k+2}\le \widetilde C\, p_k^2$ for some constant \(\widetilde C>0\), provided that \(x_0\) is chosen in a sufficiently small neighborhood of \(\ox\). It is worth noting that this estimate is weaker than the standard one-step quadratic convergence result for \(\{x_k\}\). The trade-off, however, is that the FISTA extrapolation step can move the iterates into a neighborhood of \(\ox\) more rapidly in practice.}


\begin{Corollary}[Quadratic convergence-type of the FISTA-Newton method with effective subspaces]\label{cor:quad-poly}
Under the assumptions of Theorem \ref{theo:quad-poly}, there exist $\varepsilon>0 $ and $\widetilde{C}>0$ (independent of $\ve$) such that for each starting point $x_0 \in \mathbb{B}_{\varepsilon}(\ox)$, {\bf Algorithm~\ref{algo:poly-fista}} produces a sequence $\{p_k\}$ defined by 
\eqref{eq:pk} satisfying
\begin{enumerate}
\item[{\rm {\bf (i)}}] $p_0,p_1 \le \ve$.
\item[{\rm {\bf (ii)}}]  $p_{k+2}\le \widetilde{C} p_k^2\,$ for all $k\in\N\cup \{0\}$.
\end{enumerate}
\end{Corollary}
\begin{proof} The proof is slightly modified from the one for Theorem~\ref{theo:quad-poly}. Let us recall all the constants from \eqref{eq:e-2} and define $\widetilde{C}:=\dfrac{9\ell\eta^2}{2\kappa}$, with $\eta>0$ taken from \eqref{eq:eta}. Then choose $\varepsilon>0$ 
satisfying 
\begin{equation}\label{eq:e-fista}
    0<\ve < \min \left\{1;\,\dfrac{1}{\widetilde{C}};\, {\dfrac{r}{6\eta(1+\ell)}}\right\}.
\end{equation}

 By induction, suppose that  $p_k = \max\big\{\|x_k-\ox\|;\,\|x_{k-1}-\ox\|\big\}\le \ve$ for some $k\in \N\cup \{0\}$, e.g., $p_0\le \ve$. We claim that \eqref{p:Sub} has a unique optimal solution $d_k$ and there hold the estimates $p_{k+1}\le \ve$ and $p_{k+2}\le \widetilde{C} p_k^2$. 
By \eqref{eq:fis-bound}, we have 
\begin{equation*}
\|y_k-\bar x\|+\|z_k-\bar z\|\le 3\eta p_k\le 3\eta\ve\le \frac{r}{2},
\end{equation*}
where $\eta$ is taken from \eqref{eq:eta}.
Similarly to the proof of Theorem \ref{theo:quad-poly} after \eqref{eq:y0}, the Newton step~\eqref{p:Sub} also admits a unique solution $d_k$ and 
\begin{equation}\label{eq:k+1}
    \|x_{k+1}-\ox\| \le \dfrac{\ell}{2\kappa}\|y_k-\ox\|^2 \le \widetilde{C}p_k^2 \le \widetilde{C}\ve^2 < \ve, 
\end{equation}
which yields  $p_{k+1}\le \ve$, as $\|x_k-\ox\|\le p_k\le \ve$.
The upper bound estimate $p_{k+1}\le\ve$ is enough to guarantee $(y_{k+1},z_{k+1})\in \mathbb{B}_{\frac{r}{2}}(\ox,\oz)$ due to \eqref{eq:fis-bound} {and \eqref{eq:e-fista}}, which again allows us to obtain
\begin{equation}\label{eq:k+2}
    \|x_{k+2}-\ox\|\le \widetilde{C}p_{k+1}^2.
\end{equation}
Combining \eqref{eq:k+1} and \eqref{eq:k+2} gives us that 
\begin{equation*}
\begin{aligned}
    \|x_{k+2}-\ox\| \le \widetilde{C}\max\big\{\|x_{k+1}-\ox\|^2;\, \|x_k-\ox\|^2\big\} &\le \widetilde{C}\max\Big\{\widetilde{C}^2 \|x_k-\ox\|^4;\, \widetilde{C}^2 \|x_{k-1}-\ox\|^4;\, \|x_k-\ox\|^2\Big\}\\
    &\le \widetilde{C}\max\big\{\|x_k-\ox\|^2;\,\|x_{k-1}-\ox\|^2\big\} = \widetilde{C}p_k^2,
\end{aligned}
\end{equation*}
where the last inequality is due to $\widetilde{C}^2 p_k^2 =\widetilde{C}^2\max\{\|x_k-\ox\|^2;\,\|x_{k-1}-\ox\|^2\}\le \widetilde{C}^2 \ve^2 <1$. We derive from the above inequality and \eqref{eq:k+1} that 
\begin{equation*}
    p_{k+2}\le \widetilde{C}p_k^2
\end{equation*}
as claimed. This completes the proof of the corollary.
\end{proof}

\begin{Remark}[Efficient extrapolation parameters in {\bf Algorithm~\ref{algo:poly-fista}}]
\rm
To the best of our knowledge, convergence of the iterates generated by the original FISTA in \cite{BeckTeboulle09} has not been established. However, to accelerate convergence while ensuring that the iterates approach $\ox$, we must guarantee that the sequence itself converges. This motivates the use of extrapolation schemes for which convergence of the iterates can be rigorously proved.

Several choices of the extrapolation parameters $\{\beta_k\}$ ensuring convergence of the sequence have been studied in the literature \cite{ChambolleDossal15,LiangLuoTao22}, without assuming that $\ox$ is a tilt-stable minimizer. In particular, \cite{LiangLuoTao22} proved the convergence of FISTA iterates under the parameterization
\begin{equation}\label{1stchoice}
\beta_k = \frac{t_{k-1}-1}{t_k}, 
\qquad 
t_k = \frac{p + \sqrt{\,q + 4 t_{k-1}^2\,}}{2},
\end{equation}
where $p\in (0,1)$, $q\in [p^2,(2-p)^2]$ are tunable parameters controlling the growth of $\{\beta_k\}$. 

An alternative choice was proposed in the Chambolle--Dossal variant of FISTA \cite{ChambolleDossal15}, where
\begin{equation}\label{2ndchoice}
\beta_k = \frac{k-1}{k+d},
\qquad 
t_k = \frac{k+d}{d}, \quad d>2.
\end{equation}
This rule preserves the $O(1/k^2)$ convergence rate in terms of function values while guaranteeing convergence of the iterates. In practice, both parameterizations yield trajectories that are smoother and more stable than those of the original FISTA, which often exhibits oscillatory behavior. Consequently, these variants typically achieve comparable or improved overall performance, especially on ill-conditioned problems.

\end{Remark}

\section{Effective subspaces of several polyhedral regularizers}\label{sec:exam}
\setcounter{equation}{0}





 

In our {\bf Algorithm~\ref{algo:poly0}}, the appearance of the effective subspace $L_k= \mathrm{par}\,\partial g^* (z_k)$ plays a crucial role. In this section, we illustrate concretely how these subspaces are computed in some specific polyhedral convex regularizers $g(x)$ broadly used in optimization and machine learning problems. These computations will be employed later in Section~\ref{sec:numerical}, when we experiment the performance of our {\bf Algorithms~\ref{algo:poly-ista} and~\ref{algo:poly-fista}} for different optimization models.

\begin{Proposition}[Support functions]\label{Prop:Supp} Suppose that $g=\sigma_C:\XX\to \oR$ is the support function of a closed convex set $C$. Then $\para \partial g^*(z)=\Span N_C(z)$ for any $z\in C$. 
\end{Proposition}
\begin{proof}
As $g(x)=\sigma_C(x)$, we have $g^*=\delta_C$, the  indicator function to $C$. It follows that $\partial g^*(z)=N_C(z)$. As $N_C(z)$ is a closed convex cone in $\XX$, we have
\[
\para \partial g^*(z)=N_C(z)-N_C(z)=\Span N_C(z).
\]
The proof is complete.  
\end{proof}

Any norm in $\XX$ can be represented as the support function of its {\em dual unit ball}. Particularly, we consider the case of $\ell_1$ norm, its variant known as sorted $\ell_1$ norm, and $\ell_\infty$ norm below. 

{\bf (i) $\ell_1$ norm.} Let $g(x):=\sum_{i=1}^n |x_i|$  for $x\in \R^n$. We have
\begin{equation}\label{eq:g*delta}
    g^* =(\|\cdot\|_{1})^* = (\sigma_{\mathbb{B}_{\infty}})^* = \delta_{\mathbb{B}_{\infty}},
\end{equation}
where $\B_\infty=\{v\in \R^n \,|\, |v_i|\le 1,\, i=1,\ldots, n\}$ is the $\ell_\infty$ unit ball in $\R^n$. For any $z\in \B_\infty$, it is well-known that  
\[
\partial g^* (z) = N_{\mathbb{B}_{\infty}}(z) = \left\{x\in \R^n \;\Bigg|\; 
    \begin{cases}
    x_i =0 &\text{if } |z_i|<1 \\
    x_i z_i \ge 0 &\text{if }|z_i|=1
    \end{cases}
    \right\}.
\]
Hence, we have
\begin{equation}\label{eq:spanN}
    \mathrm{par}\,\partial g^* (z) =  N_{\mathbb{B}_{\infty}}(z)-N_{\mathbb{B}_{\infty}}(z)= \{x\in \R^n \mid x_i =0 \text{ if }|z_i|<1\} = \{0\}^{J(z)} \times \R^{J(z)^c}, 
\end{equation}
where $J(z):= \{i\in \{1,\ldots,n\}\mid |z_i|<1\}$.\\

{\bf (ii) $\ell_{\infty}$ norm.} When $g$ is the $\ell_{\infty}$ norm, in the manner of \eqref{eq:g*delta} we have 
\begin{equation*}
    g^* =(\|\cdot\|_{\infty})^* = (\sigma_{\mathbb{B}_{\|\cdot\|_1}})^* = \delta_{\mathbb{B}_{\|\cdot\|_1}},
\end{equation*}
where $\mathbb{B}_{\|\cdot\|_1}:=\{v\in \R^n\,|\, \|v\|_1 \le 1\}$ is the $\ell_1$ unit ball. 
It follows that  $\partial g^* (z) = N_{\mathbb{B}_{\|\cdot\|_1}}(z)$ for each $z\in \R^n$. It follows from Proposition~\ref{Prop:Supp} that 
\begin{equation}\label{eq:par-linfty}
\mathrm{par}\,\partial g^* (z) = \mathrm{span}\,(N_{\mathbb{B}_{\|\cdot\|_1}}(z)).
\end{equation}
Throughout this section, we use the {\em sign function} as follows:
\[
\sign(\al):=\begin{cases}
        1  &\text{if }\; \al >0\\
        -1  &\text{if }\; \al<0\\
        0&\text{if }\; \al=0
    \end{cases}\quad \mbox{for any}\quad \al\in \R.
\]
Moreover, define $\sign(z):=(\sign(z_i))_{1\le i\le n}$ for any $z\in \R^n.$ According to the formula of normal cone to a lower level-set \cite[Proposition~10.3]{Rockafellar98}, we have
\begin{equation}\label{eq:N-formula}
N_{\mathbb{B}_{\|\cdot\|_1}}(z)=\cone (\partial\|\cdot\|_1(z))=
     \cone\left\{ x\in \R^n \;\Bigg|\;  \begin{cases}
        x_i \in[-1,1] &\text{if }\; z_i =0\\
        x_i = \sign(z_i) &\text{if }\; z_i \neq0
    \end{cases}
    \right\}.
\end{equation}
Combining \eqref{eq:N-formula} with \eqref{eq:par-linfty} yields the inclusion 
\begin{equation}\label{eq:par-linfty2}
    \mathrm{par}\,\partial g^* (z) \subset \left\{ x\in \R^n \mid \exists \; \alpha \in \R : 
        x_i = \alpha\, \sign(z_i)\;\text{if }\;z_i \neq 0\right\}=\R^{I(z)} \times \R\left(\sign(z_i)\right)_{i\notin I(z)},
\end{equation}
where $I(z):=\{i\mid z_i=0\}$. To justify the opposite inclusion to \eqref{eq:par-linfty2}, pick $x\in \R^n$ such that there exists $\alpha \in \R$ satisfying $x_i = \alpha\, \mathrm{sign}\,(z_i)$ for each $i\notin I(z)$. With $\alpha^+ := \max \{0,\alpha\}$ and $\alpha^- := \max\{0,-\alpha\}$, let $p,q\in \R^n$ with coordinates given by
\begin{equation*}
    (p_i,q_i):= \begin{cases}
        (x_i,0) &\text{if }z_i =0,\\
        \left(\displaystyle\sum_{j\in I(z)}|x_j| + \alpha^+,\displaystyle\sum_{j\in I(z)}|x_j| + \alpha^-\right) &\text{if }z_i >0, \\
        -\left(\displaystyle\sum_{j\in I(z)}|x_j| + \alpha^+,\displaystyle\sum_{j\in I(z)}|x_j| + \alpha^-\right) &\text{if }z_i <0,
    \end{cases}\quad \text{ for }\; i=1,\ldots,n.
\end{equation*}
Observe from \eqref{eq:N-formula} that $p,q\in N_{\mathbb{B}_{\|\cdot\|_1}}(z)$ and $x=p-q$. Consequently, the opposite inclusion to \eqref{eq:par-linfty2} holds. Hence, we have
\begin{equation}\label{eq:par-linfty3}
    \mathrm{par}\,\partial g^* (z)=\R^{I(z)} \times \R\left(\sign(z_i)\right)_{i\notin I(z)}.
\end{equation}

{\bf (iii) Sorted $\ell_1$ norm.} For any vector $x\in \R^n$, denote by $|x|_{(i)}$ the $i^{\textup{th}}$ {\em largest absolute value among coordinates} of $x$, with ties distributed by an arbitrary rule. In the Sorted L-One Penalized Estimation (SLOPE) model (see, e.g., \cite{LSTX19}), the {\em sorted $\ell_1$ norm} is defined by
\begin{equation}\label{eq:sort}
    g(x):= \sum_{i=1}^n \lambda_i |x|_{(i)}\quad \mbox{for}\quad x\in \R^n,
\end{equation}
with parameters $\lambda_1 \ge \lambda_2 \ge \ldots \ge \lambda_n \ge 0$ and $\lambda_1 >0$. Note that the OSCAR model \cite{Bondell08} is a special case of SLOPE, by choosing $\lambda_i:=w_1 + w_2 (n-i)$, $i=1,\ldots,n$ with parameters $w_1,w_2\ge 0$.

From \cite[Section~4]{LSTX19}, we know that $g^*$ is the indicator function to the following convex set 
\begin{equation}\label{eq:Clambda}
    C_{\lambda}:= \left\{z\in \R^n \;\Bigg|\; \sum_{i=1}^k |z|_{(i)} \le \sum_{i=1}^k \lambda_i\, \text{ for each }\, k\in \{1,\ldots,n\}\right\}.
\end{equation}
Hence,  $\partial g^* = N_{C_{\lambda}}$. By letting $s_k(z):=\sum_{i=1}^k |z|_{(i)}$ and $\Lambda_k:= \sum_{i=1}^k \lambda_i$, we rewrite $C_\lambda$ as $$C_{\lambda}= \displaystyle\bigcap\limits_{k=1}^n \{z\in \R^n\mid s_k (z) \le \Lambda_k\},$$ 
i.e., $C_{\lambda}$ is an intersection of finitely many lower level-sets. As $\lm_1>0$, we have $\Lambda_k >0$ for any $k\in\{1,\ldots,n\}$. Thus the Slater's condition for $C_{\lambda}$ holds at the origin. It follows from the Kuhn-Tucker theorem (see, e.g., \cite[Corollary~28.3.1]{Rockafellar70}) that $N_{C_{\lambda}}(z)$ is exactly the convex conic hull of subgradients corresponding to the active constraints at $z$, i.e., one has 
\begin{equation}\label{eq:Norsk}
N_{C_{\lambda}}(z) = \cone \left(\displaystyle\bigcup_{k\in A(z)} \partial s_k (z)\right)\quad \mbox{with}\quad A(z):=\{k\in \{1,\ldots,n\}\mid s_k (z) = \Lambda_k\}.
\end{equation}
By Proposition~\ref{Prop:Supp}, $\para\partial g^*(z)= \Span N_{C_{\lambda}}(z)$. To construct this subspace, we need to compute $\partial s_k (z)$. 

\begin{Proposition}[Subdifferential of $s_k(z)$] For any $z\in \dom g^*$ with $g$ being the sorted $\ell_1$ norm defined in \eqref{eq:sort} and for any $k\in A(z)$, define the following index sets
\begin{equation*}
    T_k := \big\{i\mid |z_i| > |z|_{(k)}\},\quad E_k := \big\{i \mid |z_i| = |z|_{(k)}\big\}, \quad{\rm and}\quad  R_k := \big\{i \mid |z_i| < |z|_{(k)}\big\}.
\end{equation*}We have 
\begin{equation*}
    \partial s_k (z) =  \begin{cases}
    \left\{v\in \mathbb{B}_{\infty} \;\Bigg|\; \|v\|_1 \le k\; \text{ and}\; v_i=\mathrm{sign}\, (z_i)\; \text{ if}\; i\in T_k
     \right\} &\text{if }|z|_{(k)}=0 \\
    \left\{v\in \mathbb{B}_{\infty} \;\Bigg|\; v_i \in  \begin{cases} \{\mathrm{sign}\, (z_i)\} &\text{ if}\; i\in T_k\\
     \R_+\mathrm{sign}\, (z_i) &\text{ if}\;i\in E_k\\
    \{0\} &\text{ if}\; i\in R_k
     \end{cases}\; \text{  and }\; \displaystyle\sum_{i\in E_k} |v_i| = k-|T_k|
     \right\} &\text{if }\; |z|_{(k)}>0,
    \end{cases}
\end{equation*}
where $|T_k|$ denotes the cardinality of the set $T_k$. Consequently, it follows that
\begin{equation*}\label{eq:spansk}
   \Span \partial s_k (z)=\begin{cases}\Span\left\{v_k, \big(e_j\big)_{j\in E_k}\right\} &\text{if }|z|_{(k)}=0\\ \Span\left\{v_k, \Bigg(e_j-\frac{\sign(z_{j})}{\sign(z_{m_k})}e_{m_k}\Bigg)_{j\in E_k\setminus\{m_k\}}\right\} &\text{if }|z|_{(k)}>0,
   \end{cases}
\end{equation*}
where $v_k$ is any subgradient in $\partial s_k (z)$ determined above, $e_j:=(0,\ldots,0, 1, 0,\ldots, 0)^T$ is the $j^{\rm th}$ unit vector in $\R^n$, and $m_k$ is the largest index in $E_k$. These formulas for  $k\in A(z)$ provide us a spanning set for the parallel space $\mathrm{par}\, \partial g^* (z)$ as follows 
\begin{equation}\label{eq:gsk}
\mathrm{par}\, \partial g^* (z) =\mathrm{span}\left(\displaystyle\bigcup_{k\in A(z)} \Span (\partial s_k (z))\right). 
\end{equation}

\end{Proposition}
\begin{proof}
First, let us claim that 
\begin{equation}\label{eq:s_k=max}
s_k (z)  = \max_{\|v\|_{\infty}\le 1,\; \|v\|_{1}\le k} \la v, z\ra.
\end{equation}
Indeed, observe that $s_k (z)=\sum_{i=1}^k |z|_{(i)} = \la \overline{v},z\ra$ with $\overline{v}\in \R^n$ given by 
\begin{equation*}
\overline{v}_i = \begin{cases} \mathrm{sign}\, (z_i) &\text{if } i\in T_k\\
    \alpha_i\, \mathrm{sign}\, (z_i) &\text{if }i\in E_k\\
    0 &\text{if } i\in R_k
     \end{cases}\; \mbox{ with }\; \al_i\ge 0, i\in E_k  \; \text{ and }\; \sum_{i\in E_k} \alpha_i = k-\big|T_k\big|.
\end{equation*}
As $\|\overline{v}\|_\infty \le 1$ and $\|\overline{v}\|_1 \le k$, we obviously have the inequality ``$\le$'' in \eqref{eq:s_k=max}. Regarding the opposite inequality, for each $v\in \R^n$  with $\|v\|_\infty \le 1$ and $\|v\|_1 \le k$, there hold the estimates
\begin{equation}\label{eq:maxvz}
\begin{aligned}
s_k (z) - \la v,z\ra  & \ge \sum_{i=1}^k |z|_{(i)} - \sum_{i=1}^n |v_i| |z_i| \\
&=\sum_{i\in T_k}(1-|v_i|)|z_i| + \left(k-|T_k|-\sum_{i\in E_k}|v_i|\right) |z|_{(k)} - \sum_{i\in R_k}|v_i||z_i|  \\
&= \sum_{i\in T_k}(1-|v_i|)\big(|z_i|-|z|_{(k)}\big) + \left(k-\sum_{i=1}^n |v_i|\right)|z|_{(k)} + \sum_{i\in R_k}|v_i|\big(|z|_{(k)}-|z_i|\big)\\
&\ge 0,
\end{aligned}
\end{equation}
which clearly ensure the inequality ``$\ge$'' in \eqref{eq:s_k=max}. 
Hence, formula \eqref{eq:s_k=max} is verified. 

By following the subdifferential formula of support function in  \cite[Corollary~23.5.3]{Rockafellar70}, we obtain  
\begin{equation}\label{eq:par=argmax}
\partial s_k (z) = \mathrm{argmax}\, \{\la v,z\ra\mid {\|v\|_{\infty} \le 1,\; \|v\|_1 \le k}\}.
\end{equation}
Any optimal vector $v^*\in \partial s_k (z)$ satisfies $s_k (z) - \la v^*,z\ra=0$. This together with the estimates from \eqref{eq:maxvz} yields 
\begin{equation}\label{eq:require-v*}
     v_i^* z_i \ge 0 \,\, \forall i,\,\, |v_i^*| = \begin{cases} 1 &\text{if } i\in T_k\\
     0 &\text{if }i\in R_k
     \end{cases}\,\,  \text{and}\,\,(k-\|v^*\|_1) |z|_{(k)} =0.
\end{equation}
If $|z|_{(k)}=0$, then $R_k = \varnothing$. It follows from \eqref{eq:require-v*} that 
\begin{equation}\label{eq:vstar}
     v^* \in \left\{v\in \mathbb{B}_{\infty} \;\Bigg|\; \|v\|_1 \le k \text{ and } v_i=\mathrm{sign}\, (z_i) \text{ if } i\in T_k
     \right\}.
\end{equation}
If $|z|_{(k)}>0$, then $\|v^*\|_1=k$. The requirements from \eqref{eq:require-v*} tell us that
\begin{equation}\label{eq:vstar2}
    v^* \in \left\{v\in \mathbb{B}_{\infty} \;\Bigg|\; v_i \in  \begin{cases} \{\mathrm{sign}\, (z_i)\} &\text{if } i\in T_k\\
    \R_+\mathrm{sign}\, (z_i) &\text{if }i\in E_k\\
    \{0\} &\text{if } i\in R_k
     \end{cases} \text{and} \sum_{i\in E_k} |v_i| = k-|T_k|
     \right\}.
\end{equation} 
This together with \eqref{eq:vstar} verifies our first formula of the proposition. 

To continue, observe from \eqref{eq:Norsk}
 and the fact that $\para \partial g^*(z)=\Span N_{C_\lm}(z)$ from Proposition~\ref{Prop:Supp} that 
\begin{equation*}
    \mathrm{par}\, \partial g^* (z) = \mathrm{span}\left(\displaystyle\bigcup_{k\in A(z)} \partial s_k (z)\right) = \mathrm{span}\left(\displaystyle\bigcup_{k\in A(z)} \mathrm{span}\,\partial s_k (z)\right),
\end{equation*}
which verifies \eqref{eq:gsk}. Denote by $\mathcal{S}_k:=\para \partial s_k (z)=\aff(\partial s_k(z)-v_k)$ for any $v_k \in \partial s_k (z)$, $k\in A(z)$. 
Let us construct the subspace $\mathrm{span}\, \big(v_k, \mathcal{S}_k\big)$ based on the two cases above:

\noindent If  $|z|_{(k)}=0$, we obtain from \eqref{eq:vstar} that  $\mathcal{S}_k =  \Big\{w \,\Big|\, w_{T_k}=0\Big\}$, which clearly implies that  
\[
\Span \partial s_k(z)=\Span\{v_k,\mathcal{S}_k\}=\Span\left\{v_k, \big(e_j\big)_{j\in E_k}\right\}. 
\]

\noindent If  $|z|_{(k)}>0$, observe from \eqref{eq:vstar2} that 
\begin{equation*}
\begin{aligned}
    \aff \partial s_k (z) &= \Big\{w \,\Big|\, w_{T_k}=\big(\mathrm{sign}\,(z)\big)_{T_k},\, w_{R_k}=0,\, \big\la w_{E_k}, \big(\mathrm{sign}\, (z)\big)_{E_k}\big\ra = k - |T_k| \Big\}, \text{ and thus }\\
    \mathcal{S}_k &= \Big\{w \,\Big|\, w_{T_k}=0,\, w_{R_k}=0,\, \big\la w_{E_k}, \big(\mathrm{sign}\, (z)\big)_{E_k}\big\ra = 0 \Big\}.
\end{aligned}
\end{equation*}
Define $m_k$ as the largest index in $E_k$. It follows that 
\[
\Span \partial s_k(z)=\Span\{v_k,\mathcal{S}_k\}=\Span\left\{v_k, \Bigg(e_j-\frac{\sign(z_{j})}{\sign(z_{m_k})}e_{m_k}\Bigg)_{j\in E_k\setminus\{m_k\}}\right\}.
\]

The formula of $\Span \partial s_k(z)$ in the above two cases verifies the second formula of the proposition. The proof is complete. 
\end{proof} 

Next, we provide an explicit computation for the effective subspace $L=\para\partial g^*(z)$, when $g$ is a composite convex function, which appears widely in many convex optimization models.

\begin{Proposition}[Composite functions] Let $K:\XX\to \YY$ be a linear operator between two Euclidean spaces and $h:\YY\to \oR$ be an l.s.c. polyhedral convex function with $K^{-1}(\dom h)\neq \emptyset$. Define the composite function $g(x)=h(Kx)$ for $x\in \XX$. Then for any $z\in \dom g^*$,  we have 
\begin{equation}\label{eq:par-com}
    \para \partial g^*(z)=K^{-1}\big(\para (\partial h^*(y)\cap \Im K)\big),
\end{equation}
for any  $y\in \YY$ satisfying 
\begin{equation}\label{eq:Sol}
    K^*y=z\quad \mbox{and}\quad g^*(z)=h^*(y).       
\end{equation}

\end{Proposition}
\begin{proof} Since $h$ is a polyhedral convex function, so is $h^*$. Thus the Fenchel-Rockafellar conjugate formula \cite[Theorem~16.3]{Rockafellar70} gives us that  
\begin{equation}\label{eq:FR}
    g^* (z) = \inf \left\{h^*(y)\mid K^*y=z\right\} \quad \forall z\in \R^n.
\end{equation}
Pick any $y$ satisfying \eqref{eq:Sol}. Note that 
Fenchel-Young's identity \eqref{def:fenyou} gives rise to the following equivalences
\begin{equation}\label{eq:EquiEx}
\begin{aligned}
    x \in \partial g^* (z) &\Leftrightarrow  \la z,x\ra = g^*(z) + g(x) \\
    &\Leftrightarrow \la K^* y,x\ra = h^* (y) + h(Kx) \\
    &\Leftrightarrow \la y, Kx\ra = h^* (y) + h(Kx) \\
    &\Leftrightarrow Kx \in \partial h^* (y).
\end{aligned}
\end{equation}
Since $Kx \in \Im K$, it implies from the above that $Kx \in \partial h^*(y) \cap \Im K$, or equivalently, $x \in K^{-1}(\partial h^*(y) \cap \Im K)$. Hence, we deduce that $\partial g^* (z)=K^{-1}\big(\partial h^* (y)\cap \Im K)$. It follows that 
\[
\partial g^* (z)-\partial g^* (z)=K^{-1}\big((\partial h^* (y)\cap \Im K)-(\partial h^* (y)\cap \Im K)\big).
\]

Define $A:=\partial g^* (z)-\partial g^* (z)$ and $B:=(\partial h^* (y)\cap \Im K)-(\partial h^* (y)\cap \Im K)\subset \Im K$. Note from the above identity that $A=K^{-1}(B)$, i.e., $K(A)=B$. We claim next that $\Span A=K^{-1}(\Span B)$. Indeed, for any $x\in \Span A$, there exist $\lm_1,\ldots,\lm_n\in \R$ and $x_1, \ldots,x_n\in A$ such that 
$x=\sum_{k=1}^n\lm_ka_k$,
which yields $Kx=\sum_{k=1}^n\lm_kKa_k\in \Span B$, as $K a_k\in B \subset \Span B$ for every $k$. This tells us that $\Span A\subset K^{-1}(\Span B)$. 

To justify the opposite inclusion, pick any $x\in K^{-1}(\Span B)$, i.e., $Kx\in \Span B$. There exist $\mu_1,\ldots,\mu_m\in \R$ and $u^1_1,u^2_1,\ldots, u^1_m,u^2_m\in \XX$  such that $Ku^1_1,Ku^2_1,\ldots,Ku^1_m,Ku^2_m\in \partial h^* (y)$ and that 
\[
Kx=\sum_{k=1}^m\mu_k(Ku_k^1-Ku_k^2).
\]
Let $w =\sum_{k=1}^m\mu_k(u_k^1-u_k^2)$, then it implies from the above that $Kx = Kw$. Hence, we derive that $x-w \in \Ker K$ or equivalently $x \in w+\Ker K$. Thus,
\[
x=\sum_{k=1}^m\mu_k(u_k^1-u_k^2)+u\quad \mbox{for some}\quad u\in \Ker K. 
\]
Pick any $q \in \Ker K$, we have $Kq = 0 \in B$, which implies $\Ker K\subset K^{-1}(B)=A$. Moreover, it follows from \eqref{eq:EquiEx} that $u^1_1,u^2_1,\ldots, u^1_m,u^2_m\in \partial g^*(z)$. Thus, we have that $u_k^1 - u_k^2 \in A$ for every $k$. Consequently, $x\in \Span A$. Since $x$ is arbitrarily chosen, we deduce that $K^{-1}(\Span B)\subset \Span A$ . This verifies that $\Span A=K^{-1}(\Span B)$. Hence, we have
\begin{align*}
\para \partial g^*(z) 
&= \Span (\partial g^*(z)- \partial g^*(z))\\
&= \Span A \\
&= K^{-1}(\Span B)\\
&= K^{-1}(\Span((\partial h^* (y)\cap \Im K)-(\partial h^* (y)\cap \Im K)))\\
&=K^{-1}\big(\para (\partial h^*(y)\cap \Im K)\big),
\end{align*}
which justifies formula \eqref{eq:par-com}. The proof is complete.
\end{proof}

\begin{Remark}{\rm In the above proposition, we suppose that $h$ is polyhedral to be consistent with the assumption that the regularizer $g$ is polyhedral throughout the paper. But formula \eqref{eq:par-com} holds for other classes of functions. We just need the Fenchel-Rockafellar conjugate formula for the composite function $h(Kx)$ \eqref{eq:FR} to be satisfied, e.g., $h$ is a continuous function; see also \cite[Section~15.3]{BauschkeCombettes17} for other qualification conditions that make this formula valid.  
}    
\end{Remark}

A particular case of composite functions widely used in time-series signal processing and data science for computing successive differences between elements (e.g., finite differences) is the total variation below. 

{\bf (iv) 1D total variation semi-norm.} Let us consider the case $g(x):= \|Dx\|_1$, where $D$ is the $(n-1)\times n$ {\em difference matrix} defined by 
\begin{equation}\label{eq:D}
    D := \begin{pmatrix}
        1 &-1 &0 &\ldots &0\\
        0 &1 &-1 &\ldots &0 \\
        \vdots &\vdots &\ddots &\ddots &\vdots \\
        0 &0 &\ldots &1 &-1
    \end{pmatrix}_{(n-1)\times n}.
\end{equation} 
With $h:=\|\cdot\|_1$ being the $\ell_1$ norm in $\R^{n-1}$, we may write $g$ as the composition $g(x) =h(Dx)$. Note from \eqref{eq:FR} that 
\[
g^*(z)=\inf\{\delta_{\B_\infty}(y)|\; D^*y=z\}=\delta_C(z)
\]
where $C:=\dom g^*=D^*(\B_\infty)$. This implies that $g(x)=\sigma_C(x)$, a particular case of Proposition~\ref{Prop:Supp}. However, the normal cone $N_C(z)$ may not have an explicit form. As $\Im D=\R^{n-1}$, the parallel space of $\partial g^*(z)$, $z\in C$, is obtained from \eqref{eq:par-com} and \eqref{eq:spanN} by
\begin{equation}\label{eq:parcom}
\para\partial g^*(z)=D^{-1}(\para N_{\B_\infty}(y))
\end{equation}
for any $y\in \B_\infty$ satisfying $D^*y=z$. It follows that 
\begin{equation}\label{eq:yz}
 \begin{cases}
        y_1 &= z_1 \\
        -y_1 + y_2&=z_2\\
        \ldots \\
         -y_{n-2} + y_{n-1}&=z_{n-1}\\
          y_{n-1}&=-z_n
    \end{cases}\quad \Longleftrightarrow\quad  \begin{cases}
        y_1 &= z_1 \\
        y_2&=z_1+z_2\\
        \ldots \\
         y_{n-1}&=z_1+\ldots+z_{n-1}\\
          -y_{n-1}&=z_n
    \end{cases}
\end{equation} 
and therefore 
\[
\dom g^* =\left\{z\in \R^n \,\Big|\, z_1 + \ldots + z_n =0\,\,{\rm and }\,\, \Bigg|\sum_{i=1}^kz_i\Bigg|\le 1,\, k=1, \ldots,n-1\right\}.
\]
From \eqref{eq:parcom} and \eqref{eq:spanN}, we have 
\begin{equation}\label{eq:Dhat}
\begin{aligned}  \para\partial g^*(z) &= \left\{x\in \R^n \mid Dx\in \para N_{\mathbb{B}_{\infty}}(y)\right\}\\
&=\big\{x\in \R^n \mid (Dx)_i =0\, \text{ if }\, |y_i|<1,\text{ for } i\in\{1,\ldots,n-1\} \big\}\\
&=\big\{x\in \R^n \mid x_i = x_{i+1}\, \text{ if }\, |y_i|<1,\text{ for } i\in \{1,\ldots,n-1\} \big\}\\
&= \big\{x\in \R^n \mid x_i = x_{i+1}\, \text{ if }\, |z_1+z_2+\ldots+z_{i}|<1, \text{ for } i\in \{1,\ldots,n-1\}\big\},
\end{aligned}
\end{equation}
whenever $z\in \dom g^*$.\\

There are many more polyhedral regularizers mentioned in the Introduction that we are able to compute their effective subspaces. However, we only choose a few special cases here to support our numerical experiments in the next section.

\section{Numerical experiments}\label{sec:numerical}
\setcounter{equation}{0}

In this section, we evaluate the performance of our proposed algorithms against several well-known first-order and second-order methods on both synthetic datasets and a widely used image dataset. The assessments presented here are of two types. The first type concerns performances in individual problems using synthetic datasets, measured in terms of suboptimality in the objective value $\|\varphi(x_k) - \varphi(\ox)\|$, the distance of the iterations to the minimizer $\|x_k - \ox\|$, and the overall runtime until the stopping criterion is satisfied. We note that the optimal value $\varphi(\ox)$, the optimal solution $\ox$, and the Newton direction $d_k$ in the subproblem are all computed using the Gurobi package~\cite{G24}. The second type of evaluation concerns the performance of the proposed algorithm in a real-life image processing application using the SMLM ISBI 2013 dataset~\cite{ISBI2013SMLM}.

\subsection{Implementation details}

For all of our numerical experiments on synthetic data, we choose the initial iterate $x_0 := 0$ and use the following \emph{relative KKT residual} stopping criterion, suggested in~\cite{LST18}, to determine the quality of an approximate optimal solution $x_k$. Specifically, we terminate the algorithms when either
\[
\dfrac{\|x_k - \prox_{\alpha_k g} (x_k - A^T(Ax_k - b))\|}{1+\|x_k\| + \|Ax_k - b \|} \le 10^{-8}
\]
or the maximum number of iterations reaches $10000$. 

Because our proposed effective-subspace Newton methods are local in nature, the methods are intended to be applied only after the iterates have entered a neighborhood of a solution. Since this neighborhood is not known in practice, we use a first-order method as a globalization phase and switch to Newton corrections based on simple diagnostics available from the iterates. To reduce the sensitivity of this heuristic switch, we include several implementation safeguards. 

The reported algorithms are hybrid proximal-gradient/Newton schemes. We first run ISTA in the Newton\_ISTA variant and FISTA in the Newton\_FISTA variant, using the fixed step size $1/\|A\|_2^2$. After this first-order phase, an effective-subspace Newton step is attempted only when the iterates appear locally stable. In our implementation, this stability is detected by requiring
\[
\|x^{k+1}-x^k\| < 10^{-2}
\]
for three consecutive iterations. The Newton correction is then combined with a backtracking line search and is accepted only if it decreases the full objective value. Specifically, the initial Newton step size is $1$, the shrinkage factor is $0.5$, and the maximum number of backtracking trials is $25$. If the Newton step is rejected, the algorithm returns to the first-order phase and waits for $8$ additional iterations before attempting another second-order correction. For the FISTA-based variant, the momentum parameter is restarted after each accepted Newton step and whenever the objective value increases.

These safeguards are important in practice: without them, the Newton correction may be applied too early or too aggressively, which can lead to divergence. The particular switching and damping strategy used here is heuristic, and other globalization or safeguarding mechanisms could also be employed to make the transition from the first-order phase to the Newton phase more robust.

All experiments were implemented in Python~3.12.7 and conducted on a MacBook Pro equipped with an Apple M2 Pro CPU and 16~GB of RAM. The code used to reproduce the numerical experiments in this section is publicly available at:
\url{https://github.com/NghiaVo99/Newton2026}.

\subsection{Least square with polyhedral regularizers}\label{Ex1}
 We implement a numerical experiment on synthetic data to solve the following regularized least-square optimization problems:
\begin{equation}\label{p:LS}
\min_{x\in \R^n}\qquad \dfrac{1}{2}\|Ax-b\|^2+\lambda g(x), \quad \lambda > 0   
\end{equation}
where the matrix $A$ has entries generated i.i.d.\ from the standard Gaussian distribution $\mathcal{N}(0,1)$, the observed signal is given by $b = A x^\star + w$ with the true signal $x^\star\in \R^n$ and $w$ denoting a small zero-mean additive Gaussian noise with variance $\sigma^2$, and $g(x)$ is a regularizer that imposes prior information on the solution $x^\star$, specified later in each experiment. Note that the effective subspaces $L=\operatorname{par}\partial g^*(z)$ used in the following experiments are readily available in the previous section. We describe the setting for each experiment in terms of the dimension of the data matrix $A \in \R^{m\times n}$, the true solution $x^\star \in \R^n$, and the tuning hyperparameters as follows:

\begin{itemize}
    \item $\ell_1$ regularizer: $(m,n) = (48,128), \|x^\star\|_{0} = 8$, that is, $x^\star$ has 8 nonzero elements with $\lambda_c = 0.1$, $\lambda = \lambda_{c}\|A^Tb\|_{\infty}$ and noise variance $\sigma^2 = 0.001$.
    \item $\ell_\infty$ regularizer: $(m,n) = (63,64), |I_{x^\star} | = 8$, where $I_{x^\star} = \{i \mid |x_i^\star| = \|x^\star\|_{\infty} \}$ with $\lambda_c = 0.1$, $\lambda = \lambda_{c}\|A^Tb\|_{\infty}$, and noise variance $\sigma^2 = 0.001$.
    \item Total variation (TV) regularizer: $(m,n) = (20,90)$ and we generate a block-constant signal to match the TV regularizer usage for piecewise-constant structure:
\[
x^\star
= \big(\underbrace{0.5,\ldots,0.5}_{30},
\underbrace{-0.3,\ldots,-0.3}_{30},
\underbrace{0.8,\ldots,0.8}_{30}\big)^\top \in \mathbb{R}^{90}.
\]
In this experiment, the tuning hyperparameters are $\lambda_c = 0.3, \lambda = \lambda_c \|A^Tb\|_{\infty}$ and noise variance $\sigma^2 = 0.001$.
    \item OSCAR regularizer: $(m,n) = (300,300)$. We generate a design matrix \(A \in \mathbb{R}^{300 \times 300} \) whose columns $i$ and $j$ have covariance
\[
\Sigma_{ij} := \rho^{|i-j|}, \quad \rho = 0.7.
\]
We then standardize the columns of \( A \) and define the solution coefficient vector as
\[
x^\star = (
\underbrace{0, \dots, 0}_{0.15n}, \;
\underbrace{3.0, \dots, 3.0}_{0.05n}, \;
\underbrace{0, \dots, 0}_{0.25n}, \;
\underbrace{-4.0, \dots, -4.0}_{0.05n}, \;
\underbrace{0, \dots, 0}_{0.20n}, \;
\underbrace{6.0, \dots, 6.0}_{0.05n}, \;
\underbrace{0, \dots, 0}_{0.25n})^\top \in \R^{300}.
\]
These block-wise coefficients emulate correlated features that share identical magnitudes within each group, which is similar to the setting used in \cite{zeng2014ordered}. 
The response vector $b$ is added with a small noise with $\sigma^2 = 0.01$. In this experiment, the tuning hyperparameters are $\lambda_c = 10^{-6}, \lambda_1 = \lambda_c \|A^Tb\|_{\infty}$ and $\lambda_2 = \lambda_1$.
\end{itemize}
We compare the performance of our proposed {\bf Algorithm~\ref{algo:poly-ista}} and {\bf Algorithm~\ref{algo:poly-fista}} with the following first-order and second-order methods:
\begin{enumerate}
\item The Iterative Shrinkage-Thresholding Algorithm (ISTA) in \cite{BeckTeboulle09}.
\item The Fast Iterative Shrinkage-Thresholding Algorithm (FISTA) in \cite{BeckTeboulle09}.
\item The Semismooth Newton Augmented Lagrangian method (SSNAL) in \cite{LST18}.
\item The Coderivative-based Generalized Regularized Newton Method (GRNM) in \cite{KhanhMordukhovichPhat24}.
\item The Globalized SCD Semismooth$^{*}$ Newton method (GSSN) in \cite{GfrererCOA25}.
\end{enumerate}

In the plots, we refer to \textbf{Algorithm~\ref{algo:poly-ista}} as \(\mathrm{Newton\_ISTA}\) and \textbf{Algorithm~\ref{algo:poly-fista}} as \(\mathrm{Newton\_FISTA}\), respectively. Their variants with the well-known {\em backtracking line-search} \cite{BeckTeboulle09} are denoted by \(\mathrm{Newton\_BT\_ISTA}\) and \(\mathrm{Newton\_BT\_FISTA}\), respectively.


\begin{figure}[H]
\centering
\includegraphics[width=1.1\linewidth,keepaspectratio]{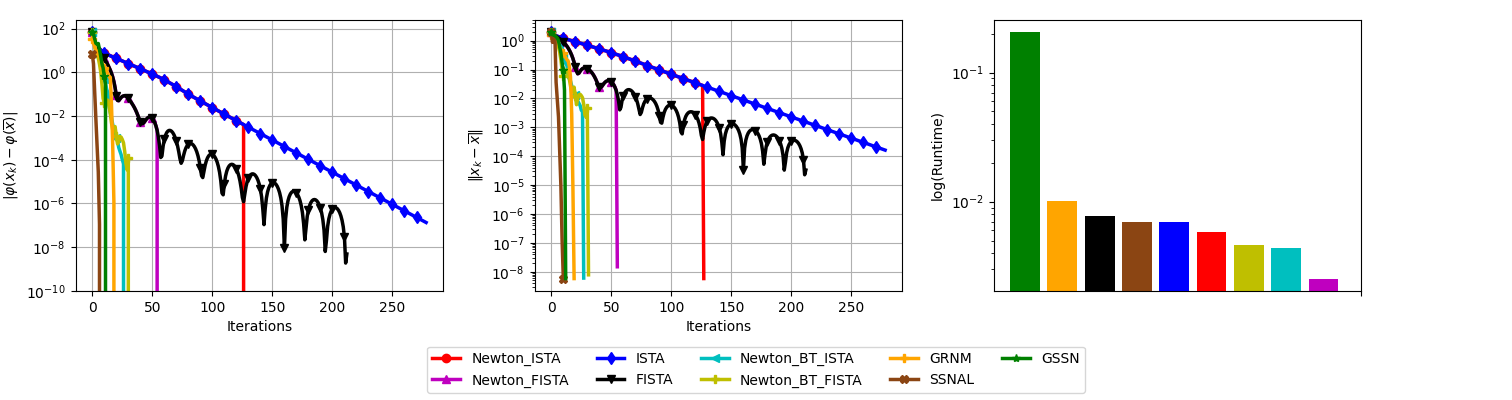} 
\caption{$\ell_1$ norm.}
\end{figure}

In the Lasso experiment with $g(x)=\|x\|_1$, Figure 1 demonstrates that our  proposed effective subspace Newton methods, especially {\bf Algorithm~\ref{algo:poly-fista}}, enter a quadratic local region rapidly, attaining the prescribed KKT tolerance in substantially fewer iterations and shorter wall time than ISTA and FISTA. Compared to the coderivative-based Newton method GRNM, our method, despite exhibiting a slower decrease in the objective value per iteration, achieves faster overall runtime once the effective subspace is identified, reflecting its lower per-iteration complexity. Notably, GSSN incurs extra computational overhead from the trial backtracking loops that safeguard the forward–backward envelope decrease, see \cite[Algorithm 1]{GfrererCOA25}, which explains its longer runtime despite significant improvement per iteration. Moreover, a comparison with the semismooth Newton augmented-Lagrangian (SSNAL) solver shows that our framework attains a comparable Newton-type local convergence speed for Lasso without resorting to augmented Lagrangian subproblems, which require evaluating the generalized Jacobian of the nonsmooth proximal mapping. In contrast, our methods rely solely on explicitly computable proximal mappings, together with the parallel-subspace construction in Section~\ref{sec:exam}.
\begin{figure}[H]
\centering
\includegraphics[width=1.1\linewidth,keepaspectratio]{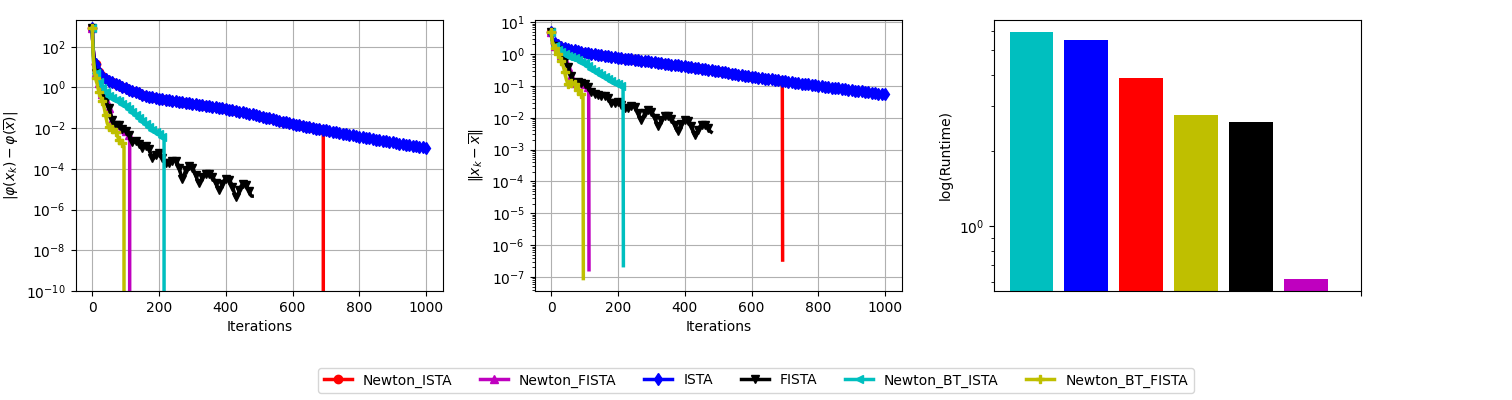} 
\caption{$\ell_\infty$ norm.}
\end{figure}
Next, we compare the performance of our algorithm with ISTA, FISTA, and SSNAL \cite{LSTX19} for solving the OSCAR model. We are not able to implement the GRNM and GSSN for this model. 
\begin{figure}[H]
\centering
\includegraphics[width=1.1\linewidth,keepaspectratio]{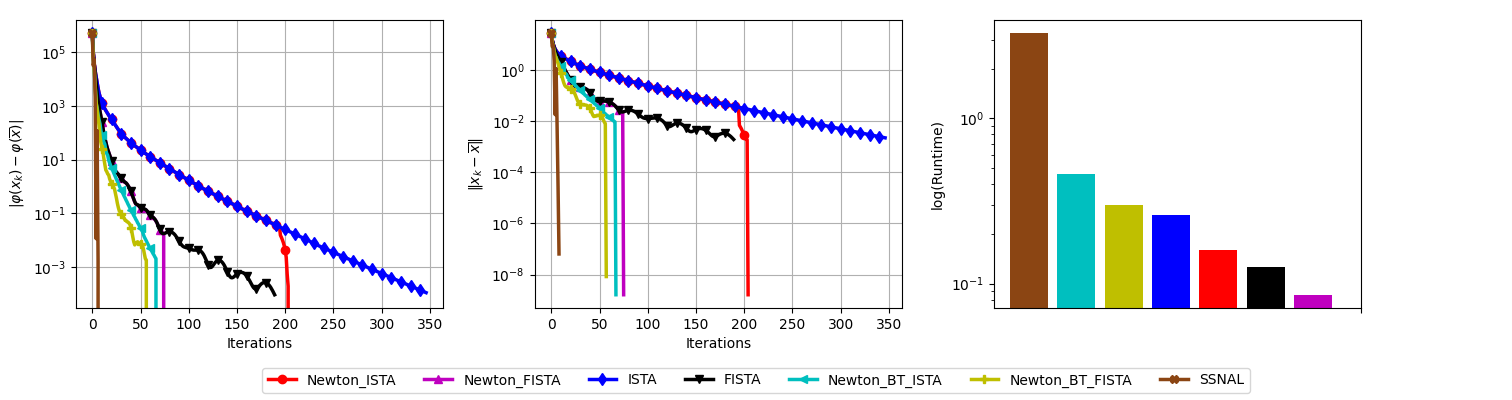} 
\caption{OSCAR regularizer.}
\end{figure}

Across the $\ell_\infty$, 1D total variation, and OSCAR experiments, the observed behavior is consistent with the Lasso benchmark. Our method identifies the effective structure at an early stage and subsequently exhibits the quadratic convergence profile, supporting our convergence analysis in Section~\ref{sec:polyhedral}. In all three settings, it attains the set KKT tolerance using significantly fewer iterations and reduced CPU runtime relative to standard first-order baselines, {thereby outputting} a high-accuracy approximate solution and avoids the extended plateaus or oscillation of first-order algorithms.

A fundamental difference between our proposed \textbf{Algorithm~\ref{algo:poly-ista}} and \textbf{Algorithm~\ref{algo:poly-fista}} and other second-order schemes is that our methods are local in nature: the quadratic convergence rate is guaranteed only when the sequence of iterates is sufficiently close to an optimal solution. Consequently, in our framework it is crucial to rely on the global convergence properties of the underlying first-order schemes. A purely heuristic switching rule that triggers a Newton step too early may cause the method to diverge. The advantage of this design is that, away from the solution, the per-iteration complexity essentially coincides with that of a first-order method, which explains the reduced computational time observed in our experiments. The trade-off is that the switching tolerance for activating the Newton step must be tuned appropriately to balance convergence speed and robustness. In contrast, GRNM, GSSN, and SSNAL are globally convergent second-order methods and therefore require more elaborate algorithmic designs and incur higher computational complexity. However, since they do not rely on any switching condition to activate the Newton step, they can be more robust in several situations.

\begin{figure}[H]
\centering
\includegraphics[width=1\linewidth,keepaspectratio]{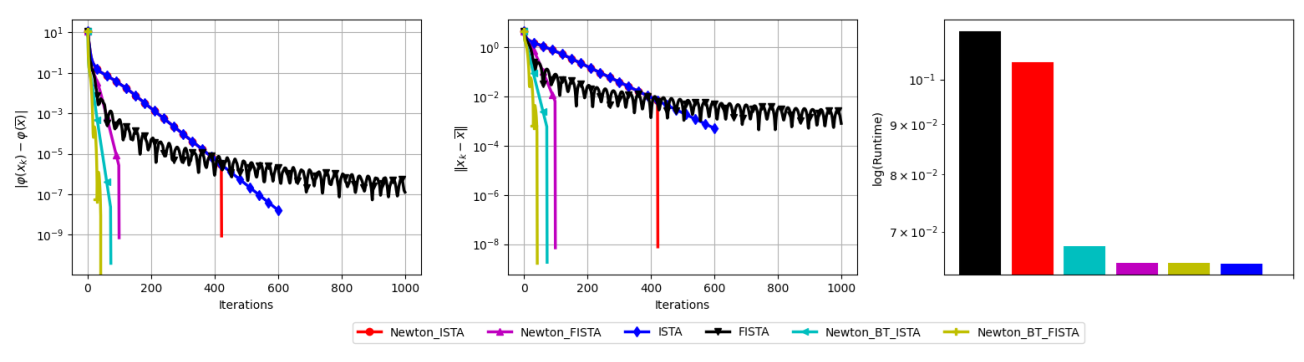} 
\caption{1D total variation semi-norm.}
\end{figure}



\subsection{Further performance benchmarks with Benchopt optimizers}
In this subsection, we establish more performance benchmarks between our proposed methods and those available in Benchopt framework
\cite{moreau2022benchopt}, which is a standardized environment for comparing
optimization algorithms on common benchmark problems, making it useful for
evaluating both convergence behavior and computational time across different
solvers.

Specifically, Tables~\ref{tab:runtime_lasso}--\ref{tab:runtime_tv} report the CPU times required by all tested methods to attain objective values within $10^{-8}$ of the best obtained value. {All reported numerical results are based on a single run with subproblem time included, smaller values indicating faster performance. The best value in each column is highlighted in bold, while the second-best value is shown in italics.} Meanwhile, the Figures~\ref{fig:lasso_bench}, \ref{fig:oscar_bench}, and
\ref{fig:tv_bench}, show the suboptimality decay with respect to the
iteration number on a log--log scale. Thus, the figures mainly describe
iteration-wise convergence behavior, whereas the tables provide a wall-clock
comparison.

From Figures~\ref{fig:lasso_bench}, \ref{fig:oscar_bench}, and
\ref{fig:tv_bench}, we observe that the Newton-corrected variants generally improve the
iteration-wise behavior of the corresponding ISTA and FISTA schemes. In many
runs, \textbf{Algorithm~\ref{algo:poly-ista}} and \textbf{Algorithm~\ref{algo:poly-fista}} follow the same trend as ISTA and FISTA during
the early iterations, and then decrease the suboptimality gap more rapidly after
the iterates enter a local regime where the effective subspace becomes useful.

For the Lasso benchmark,
Table~\ref{tab:runtime_lasso} indicates that specialized solvers such as
coordinate descent, Celer, \texttt{skglm}, and \texttt{sklearn} are faster in
wall-clock time than the \textbf{Algorithm~\ref{algo:poly-ista}} and \textbf{Algorithm~\ref{algo:poly-fista}} implementations tested
here. This is reasonable, since these solvers are highly optimized first-order methods and
often exploit model-specific tools such as coordinate-wise updates, screening
rules, and efficient low-level implementations. Therefore, while the Newton
correction can improve ISTA and FISTA as iterative schemes, {the proposed methods do not outperform} those solvers in CPU time.

\begin{table}[H]
\centering
\begin{threeparttable}
\small
\resizebox{\textwidth}{!}{%
\begin{tabular}{@{}llrrrr@{}}
\toprule
Problem & Solver
& \multicolumn{4}{c}{Hyperparameters} \\
\cmidrule(lr){3-6}
& & $\lambda=0.1,\ \rho=0$
  & $\lambda=0.1,\ \rho=0.6$
  & $\lambda=0.01,\ \rho=0$
  & $\lambda=0.01,\ \rho=0.6$ \\
\midrule
\multirow{8}{*}{%
\begin{tabular}[c]{@{}c@{}}
Lasso\\
$(m=500,\ n=600)$
\end{tabular}}
& Celer                   & \textit{0.0045} & 0.0050          & \textit{0.0400} & \textit{0.0390} \\
& Coordinate Descent (CD) & \textbf{0.0024} & \textbf{0.0047} & 0.0500          & 0.0420          \\
& FISTA                   & 0.0780          & 0.1700          & 0.2700          & 0.3700          \\
& ISTA                    & 0.0920          & 0.1500          & 0.3700          & 0.8000          \\
& ISTA\_Newton            & 0.0840          & 0.1100          & 0.1900          & 2.8000          \\
& FISTA\_Newton           & 0.1000          & 0.0830          & 0.7000          & 1.7000          \\
& skglm                   & 0.0060          & 0.0080          & \textbf{0.0350} & \textit{0.0390} \\
& sklearn                 & 0.0047          & \textit{0.0048} & 0.0470          & \textbf{0.0330} \\
\bottomrule
\end{tabular}%
}
\end{threeparttable}
\caption{CPU time in seconds for the Lasso benchmark.}
\label{tab:runtime_lasso}
\end{table}

\begin{table}[H]
\centering
\begin{threeparttable}
\small
\begin{tabular}{@{}llrr@{}}
\toprule
Problem & Solver
& \multicolumn{2}{c}{Hyperparameters} \\
\cmidrule(lr){3-4}
& & $w_1=10^{-3},\ w_2=10^{-4},\ \rho=0$
  & $w_1=10^{-3},\ w_2=10^{-4},\ \rho=0.6$ \\
\midrule
\multirow{6}{*}{%
\begin{tabular}[c]{@{}c@{}}
OSCAR\\
$(m=500,\ n=1000)$
\end{tabular}}
& ADMM          & 1.100           & 2.500           \\
& Newt-ALM      & 0.290           & \textit{0.250}  \\
& FISTA         & \textit{0.140}  & \textbf{0.240}  \\
& ISTA          & \textbf{0.130}  & 0.410           \\
& ISTA\_Newton  & 0.180           & 0.430           \\
& FISTA\_Newton & 0.180           & \textbf{0.240}  \\
\bottomrule
\end{tabular}
\end{threeparttable}
\caption{CPU time in seconds for the OSCAR benchmark.}
\label{tab:runtime_oscar}
\end{table}

For the OSCAR and one-dimensional total-variation benchmarks, the comparison is
somewhat different. Table~\ref{tab:runtime_oscar} shows that the Newton variants
are competitive with ISTA and FISTA and are faster than ADMM in the tested OSCAR
instances. Figure~\ref{fig:oscar_bench} also suggests that the Newton correction
can improve the reduction of the suboptimality gap after the initial stage.
Similarly, Table~\ref{tab:runtime_tv} shows that Newton ISTA and Newton FISTA can
reduce the CPU time of their first-order counterparts for the tested
total-variation instances, while Figure~\ref{fig:tv_bench} shows a faster
suboptimality decrease in terms of iterations.

Although the proposed methods show promising improvements over the recently developed second-order methods considered in the previous subsection, the Benchopt experiments indicate that effective-subspace acceleration does not automatically lead to the best wall-clock performance when compared with highly optimized state-of-the-art solvers. In particular, for classical problems such as Lasso, solvers based on coordinate descent, Celer, \texttt{skglm}, and \texttt{sklearn} remain more competitive in CPU time on the tested instances. Therefore, the purpose of these experiments is not to claim superiority over specialized production-level solvers, but rather to demonstrate that the proposed Newton-type correction can serve as a general acceleration mechanism for proximal-gradient methods such as ISTA and FISTA under polyhedral regularization.

\begin{table}[H]
\centering
\begin{threeparttable}
\small
\begin{tabular}{@{}llrr@{}}
\toprule
Problem & Solver
& \multicolumn{2}{c}{Hyperparameters} \\
\cmidrule(lr){3-4}
& & $\lambda=0.1$ & $\lambda=0.01$ \\
\midrule
\multirow{6}{*}{%
\begin{tabular}[c]{@{}c@{}}
1D Total Variation\\
$(m=500,\ n=1000)$
\end{tabular}}
& Celer         & \textit{0.022} & \textbf{0.034} \\
& ISTA          & 0.210          & 0.250          \\
& FISTA         & 0.280          & 0.250          \\
& ISTA\_Newton  & 0.075          & 0.099          \\
& FISTA\_Newton & 0.039          & 0.038          \\
& skglm         & \textbf{0.016} & \textit{0.036} \\
\bottomrule
\end{tabular}
\end{threeparttable}
\caption{CPU time in seconds for the 1D total variation benchmark.}
\label{tab:runtime_tv}
\end{table}

\begin{figure}[H]
  \centering

  \begin{subfigure}{0.49\textwidth}
    \centering
    \includegraphics[width=0.95\linewidth]{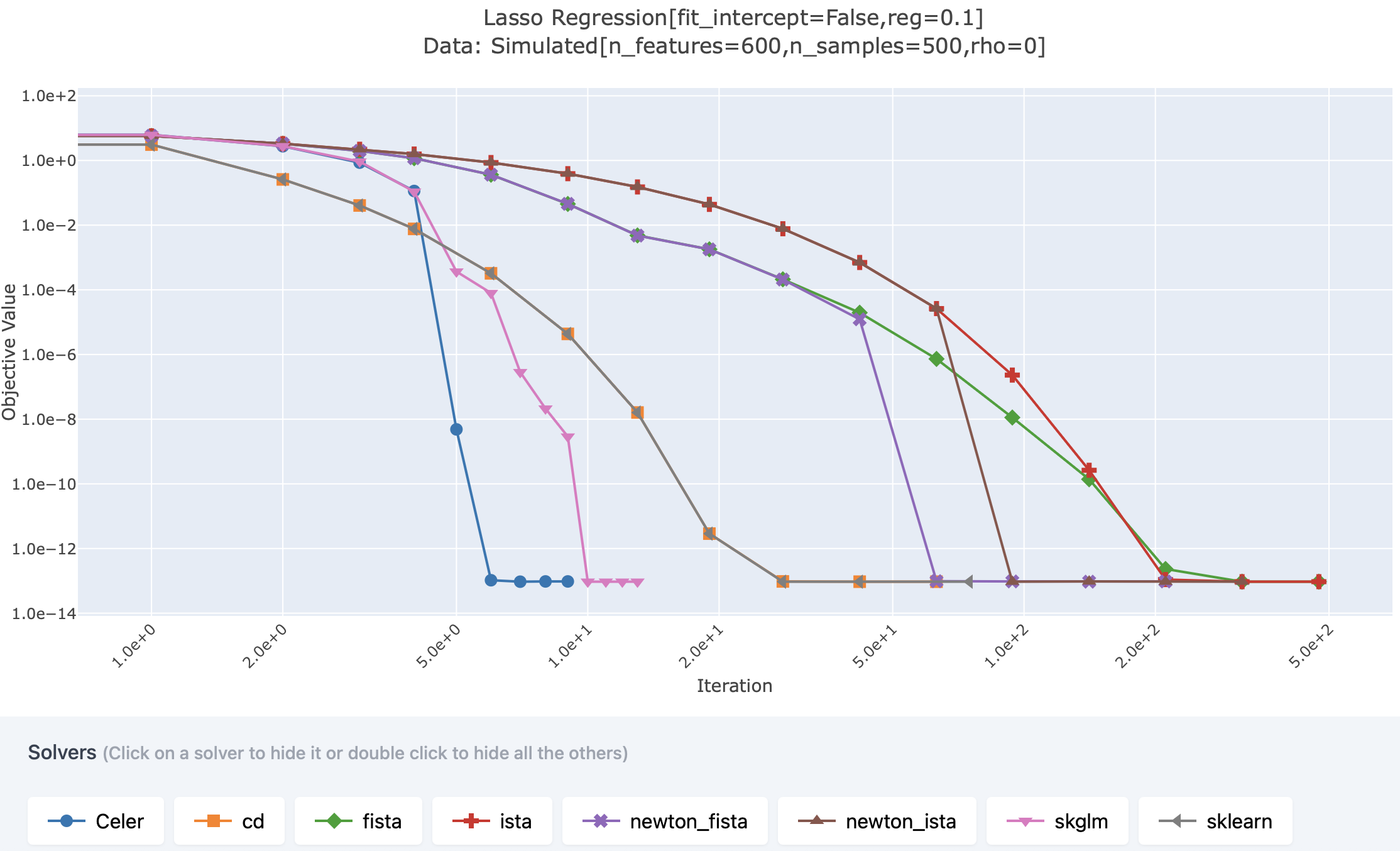}
    \caption{$\lambda = 0.1,\rho = 0$}
  \end{subfigure}
  \hfill
  \begin{subfigure}{0.49\textwidth}
    \centering
    \includegraphics[width=0.95\linewidth]{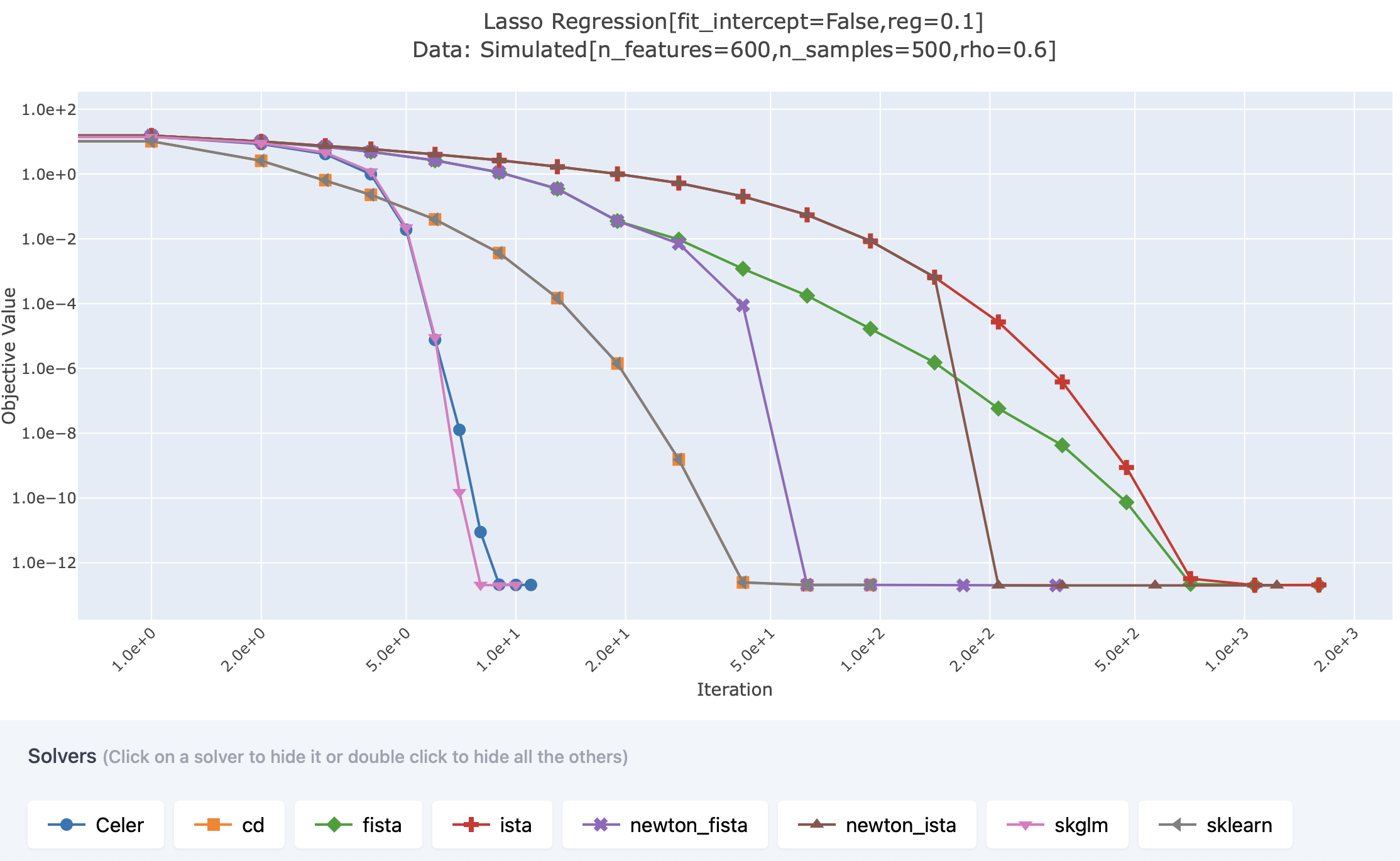}
    \caption{$\lambda = 0.1, \rho = 0.6$}
  \end{subfigure}

  \vspace{0.5em}

  \begin{subfigure}{0.49\textwidth}
    \centering
    \includegraphics[width=0.95\linewidth]{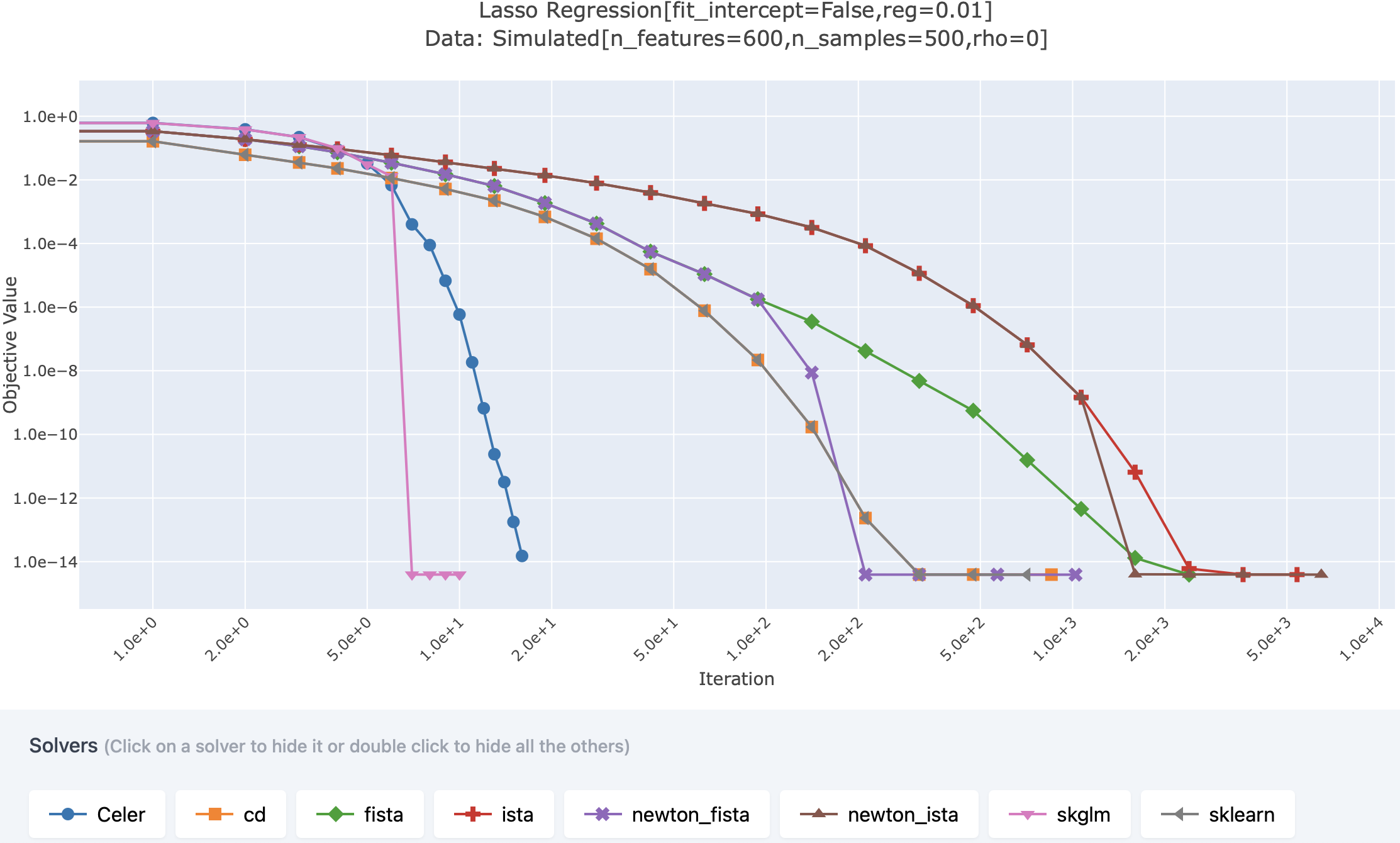}
    \caption{$\lambda = 0.01, \rho = 0$}
  \end{subfigure}
  \hfill
  \begin{subfigure}{0.49\textwidth}
    \centering
    \includegraphics[width=0.95\linewidth]{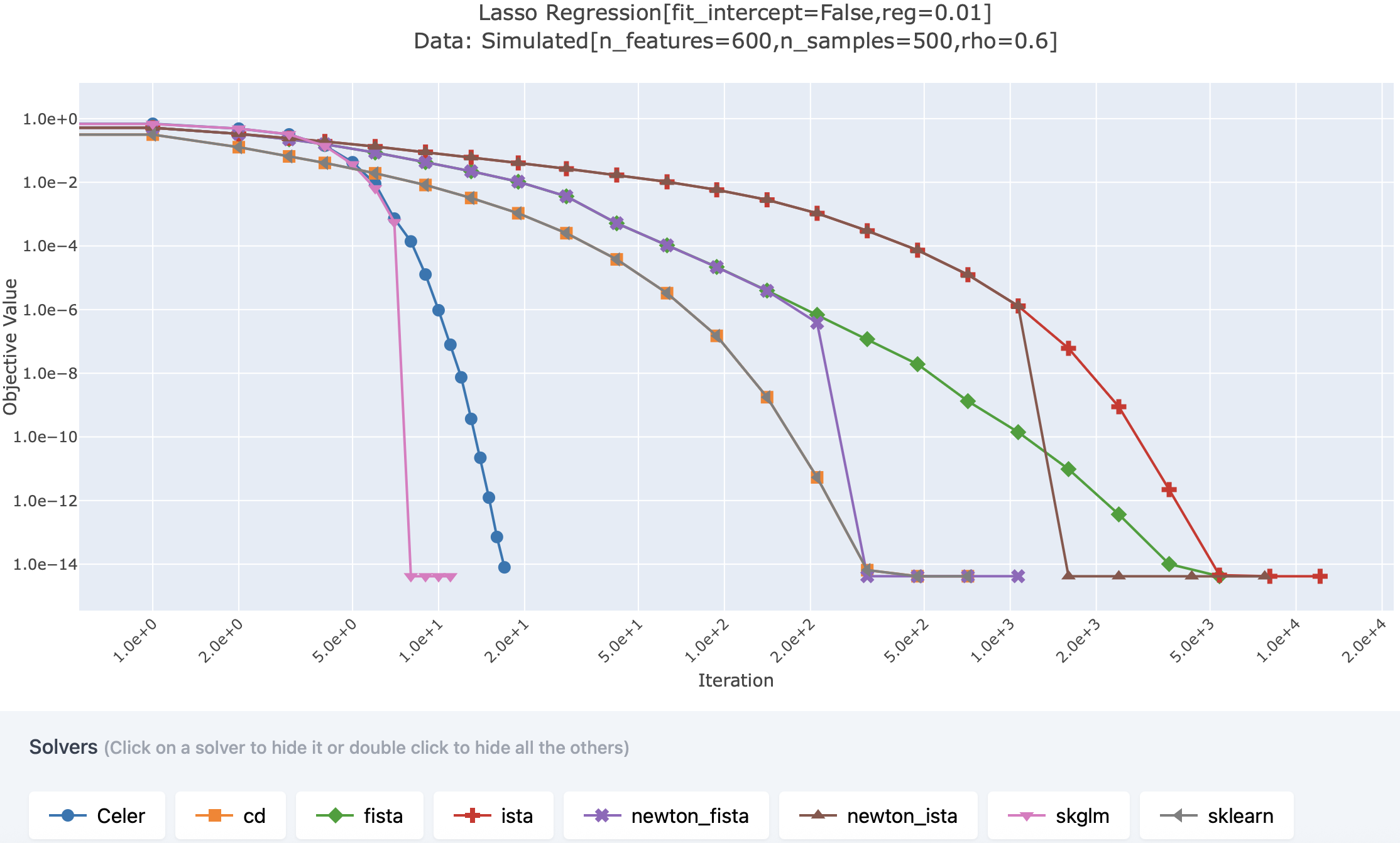}
    \caption{$\lambda = 0.01, \rho = 0.6$}
  \end{subfigure}

  \caption{Benchmark Lasso of {\bf Algorithm~\ref{algo:poly-ista}} and {\bf Algorithm~\ref{algo:poly-fista}} versus available solvers in Benchopt with different simulated datasets.}
  \label{fig:lasso_bench}
\end{figure}

\begin{figure}[H]
  \centering

  \begin{subfigure}{0.49\textwidth}
    \centering
    \includegraphics[width=0.95\linewidth]{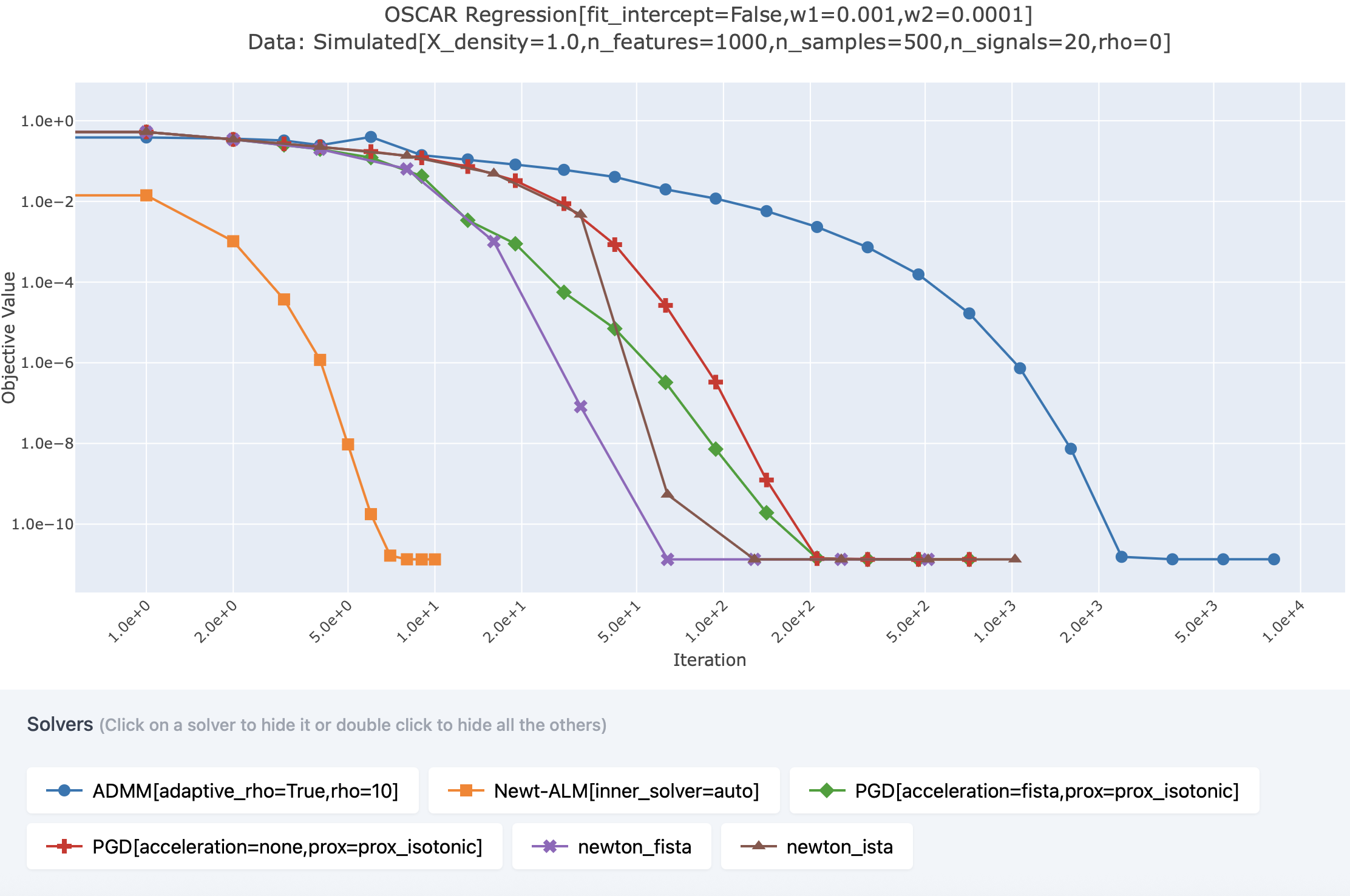}
    \caption{$w_1 = 1e-3, w_2 = 1e-4, \rho = 0$}
  \end{subfigure}
  \hfill
  \begin{subfigure}{0.49\textwidth}
    \centering
    \includegraphics[width=0.95\linewidth]{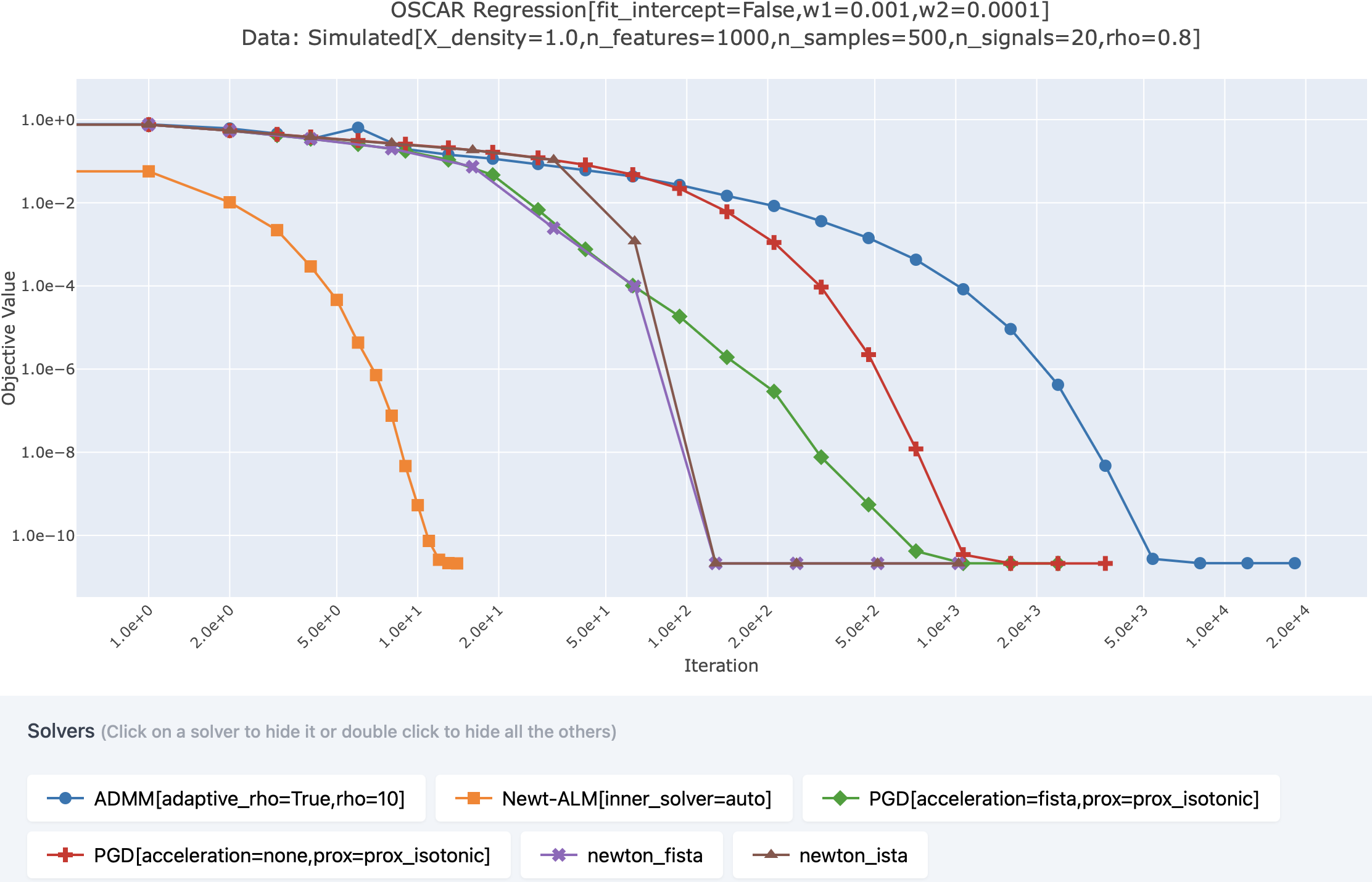}
    \caption{$w_1 = 1e-3, w_2 = 1e-4, \rho = 0.6$}
  \end{subfigure}

  \caption{Benchmark OSCAR of {\bf Algorithm~\ref{algo:poly-ista}} and {\bf Algorithm~\ref{algo:poly-fista}} versus available solvers in Benchopt with different simulated datasets.}
  \label{fig:oscar_bench}
\end{figure}

\begin{figure}[H]
  \centering

  \begin{subfigure}{0.49\textwidth}
    \centering
    \includegraphics[width=0.95\linewidth]{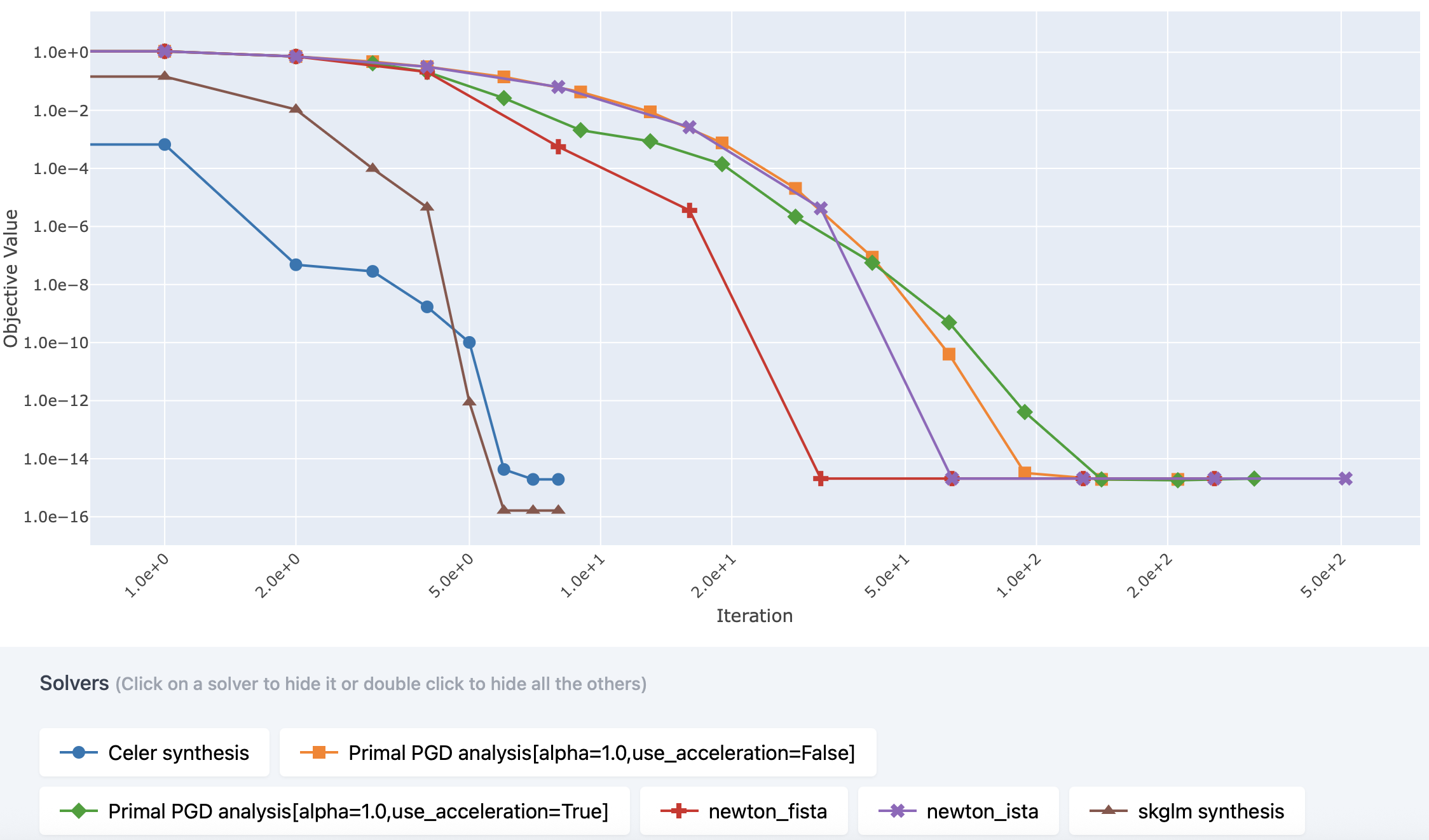}
    \caption{$\lambda = 0.1$}
  \end{subfigure}
  \hfill
  \begin{subfigure}{0.49\textwidth}
    \centering
    \includegraphics[width=0.95\linewidth]{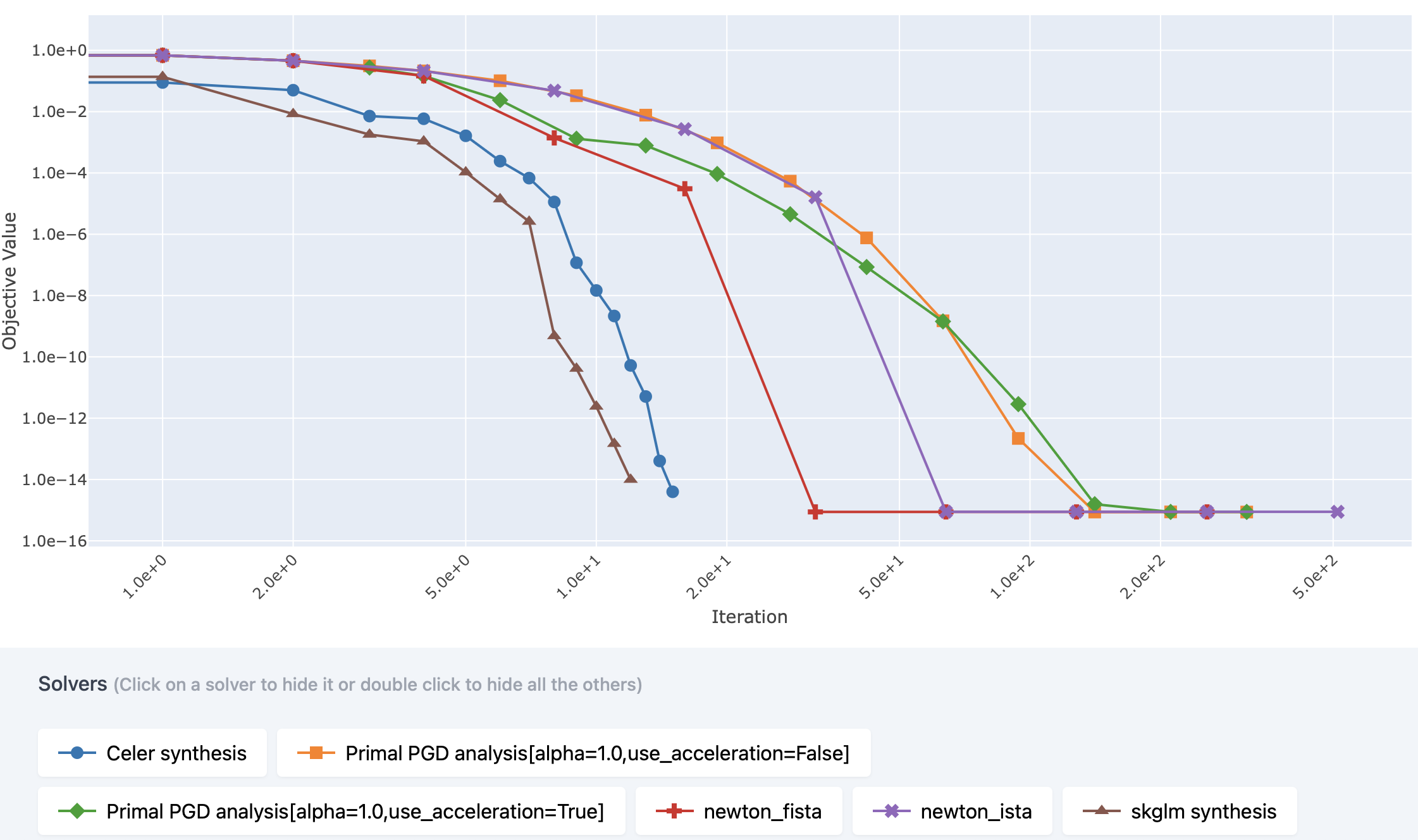}
    \caption{$\lambda = 0.01$}
  \end{subfigure}

  \caption{Benchmark 1D Total Variation problem of {\bf Algorithm~\ref{algo:poly-ista}} and {\bf Algorithm~\ref{algo:poly-fista}} versus available solvers in Benchopt with different simulated datasets.}
  \label{fig:tv_bench}
\end{figure}

\subsection{Sparse nonnegative image super-resolution with Poisson data}

In this experiment, {as a qualitative proof of concept on the application of the proposed methods}, we address the problem of reconstructing a high-resolution nonnegative image from low-resolution Poisson-corrupted measurements, a scenario commonly encountered in fluorescence microscopy and photon-limited imaging \cite{Lazzaretti2021, RSCL22}. 
The dataset for this task is the SMLM ISBI 2013 dataset, which is composed of 361 images representing 8 tubes of $30$nm diameter. The size of each image is $64 \times 64$ where each pixel is of size $100 \times 100$nm . We localize the molecules on a $256 \times 256$ pixel image corresponding to a factor $L = 4$, where the size of each pixel is thus $25 \times 25$nm. The Gaussian Point Spread Function has full width at half maximum (FWHM) equals $258.2$nm. In Figure~\ref{fig:frame}, we demonstrate several frames from the data set. 
\begin{figure}[h]
\centering
\includegraphics[width=0.85\linewidth]{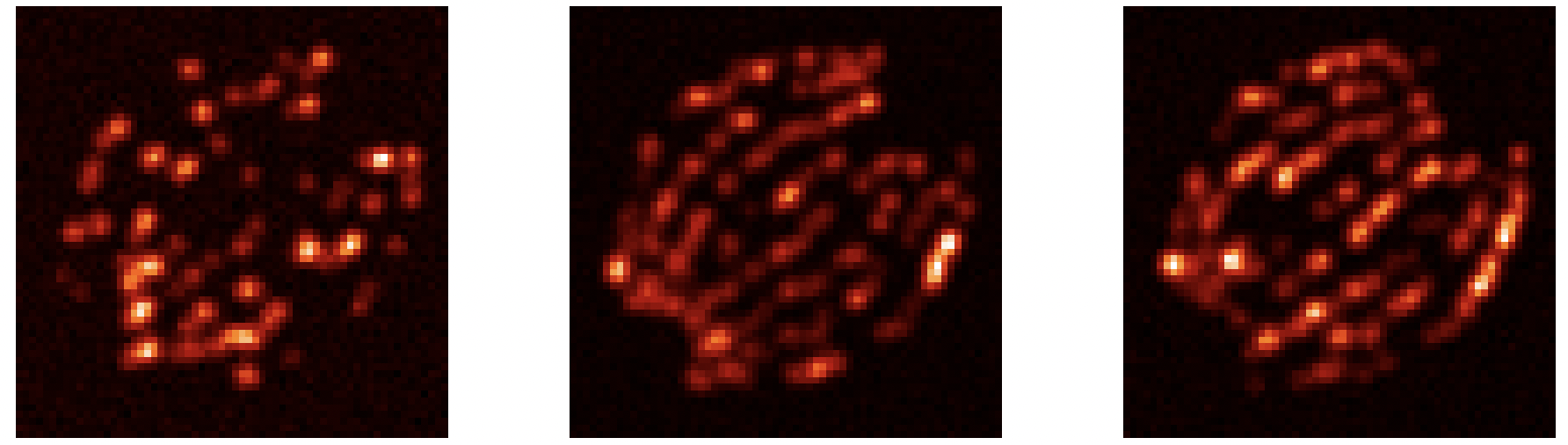}
  \caption{Images in SMLM ISBI 2013 dataset.}
  \label{fig:frame}
\end{figure}

We now describe the optimization model for the reconstruction task. Given a low-resolution and noisy observation \(y \in \mathbb{R}^m_{+}\), we seek a sparse, nonnegative high-resolution image \(x \in \mathbb{R}^n_{+}\) satisfying the forward process
\[
y \;=\; \mathcal{P}(MHx + b),
\] where $H \in \mathbb{R}^{n \times n}$ is a known blurring/convolution operator corresponding to the point spread function (PSF), $M \in \mathbb{R}^{m \times n}$ is a downsampling operator with factor $q>1$, \(\mathcal{P}(w)\) denotes a realization of an \(m\)-dimensional Poisson random vector with mean parameter \(w \in \mathbb{R}^m_{+}\). The vector $b \in \mathbb{R}^m_{>0}$ represents a constant background.

Following the standard Bayesian Maximum A Posteriori (MAP), one finds that the data fidelity term corresponding to the negative log-likelihood of
 the Poisson p.d.f. is the Kullback-Leibler (KL) divergence. Hence, the sparse nonnegative image super-resolution task can be formulated as minimizing a constrained composite optimization problem that combines a Poisson data-fidelity term, expressed via the {\em Kullback--Leibler} (KL) {\em divergence}, with the $\ell_1$ sparsity-promoting regularizer:
\begin{equation}
    \min_{x \ge 0} \;
    D_{\mathrm{KL}}(y \,\|\, M H x + b)
    + \lambda \|x\|_1,
    \label{eq:poisson_sr_problem}
\end{equation}
where $y \in \mathbb{R}^m_{+}$ denotes the observed photon counts and $\lambda > 0$ is a regularization parameter controlling the trade-off between data fidelity and sparsity. The KL divergence term is defined as
\begin{equation*}
    f(x):=D_{\mathrm{KL}}(y \,\|\, M H x + b)
    = \sum_{i=1}^{m} \left[
        y_i \log \!\left( \frac{y_i}{(M H x)_i + b_i} \right)
        - y_i + (M H x)_i + b_i
    \right].
\end{equation*}
Although this function does not have full domain as required in our Theorem~\ref{theo:quad-poly}, its domain is an open convex set. Adapting our proof to this case is straightforward but technically involved, so we omit the details.  

Note that problem \eqref{eq:poisson_sr_problem} can be reformulated in the form \eqref{p:CPintro} as an unconstrained optimization problem with the polyhedral regularizer  $g(x) := \lambda \|x\|_1 + \delta_{\R^n_{+}}(x)$.
Next, we briefly discuss the necessary computation steps for solving \eqref{eq:poisson_sr_problem} using our proposed Newton method with effective subspaces. Firstly, the Hessian of the KL divergence admits the formulation
\begin{equation*}
\nabla^2 f(x)=
H^{\top} M^{\top} \,
\Diag\!\left( \frac{y}{\,( M H x + b)^{2}} \right)
\, M H ,
\end{equation*}
{where the division and the ``square'' are taken element-wise}. It is also locally Lipschitz continuous around any point in its domain.
{In addition}, the proximal operator of $g$ can be computed explicitly as the projection of the soft-thresholded point onto $\mathbb{R}^n_{+}$:
\begin{equation*}
    {\rm \textbf{prox}}_{\al g}(v)
    = P_{\mathbb{R}^n_{+}}\!\left(
        {\rm \textbf{prox}}_{\al \lambda \|\cdot\|_1}(v)
      \right)
    = \max\,\!\left( 0,\, \operatorname{sign}\,(v) \odot (|v| - \al \lambda) \right),
    \label{eq:prox_nonneg_l1}
\end{equation*}
where $\odot$ denotes the componentwise multiplication. The Legendre-Fenchel conjugate of $g(x)$ is
\[
g^{*}(z)
= \sup_{x \ge 0}\, \{\langle z, x\rangle - \lambda\|x\|_{1}\}
= \sum_{i=1}^{n} \sup_{x_{i} \ge 0}\, (z_{i} - \lambda)x_{i} = \delta_{(-\infty,\lambda]^{n}}(z).
\]
For any \( z \in (-\infty,\lambda]^{n} \), denote by \( I(z) := \{\,i\in \{1,\ldots,n\} \mid z_{i}=\lambda\,\} \) the active index set. The subdifferential \( \partial g^{*}(z) \)
coincides with the normal cone to the box \( C = (-\infty,\lambda]^{n} \), which has the representation
\[
N_{C}(z)
    = \{\, v \in \mathbb{R}^{n} \mid 
       v_{i} \ge 0\ \text{if}\ z_{i}=\lambda,\
       v_{i}=0\ \text{if}\ z_{i}<\lambda \,\} = \{\, v \in \mathbb{R}^{n} \mid v_{I(z)} \ge 0,\ v_{I(z)^{c}} = 0 \,\}.
\]
The corresponding effective subspace, which defines the constraint in
the Newton step, is obtained easily as
\[
\operatorname{par}\partial g^{*}(z)
    = N_{C}(z) - N_{C}(z)
    = \mathbb{R}^{I(z)} \times \{0\}^{I(z)^{c}}.
\]
In this experiment, the starting iteration is $x_0 = H^{\top} M^{\top} y$ and the tuning hyperparameter is $\lambda = 0.5\|\max\{\nabla f(0),0\}\|_{\infty}$. Since the gradient $\nabla f$ is not globally Lipschitz continuous, to the best of our knowledge there is no solid theoretical guarantee for using ISTA or FISTA with a constant stepsize $\alpha>0$, especially concerning the convergence rate. However, $\nabla f$ is still locally Lipschitz continuous, and thus, when equipped with a backtracking line-search, both ISTA and FISTA schemes converge \cite{BelloCruzNghia16}. The reconstruction results are illustrated in Figure~\ref{fig:frame1}, in which the proposed {\bf Algorithm~\ref{algo:poly-fista}} with backtracking line-search successfully recovers fine structural details of the image.

\begin{figure}[h]
  \centering
  \begin{subfigure}{0.333\textwidth}\centering
    \includegraphics[width=0.7\linewidth]{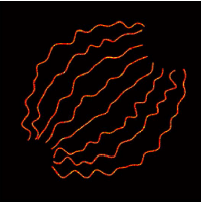}
  \end{subfigure}%
  \begin{subfigure}{0.333\textwidth}\centering
    \includegraphics[width=0.7\linewidth]{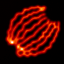}
  \end{subfigure}%
  \begin{subfigure}{0.333\textwidth}\centering
    \includegraphics[width=0.7\linewidth]{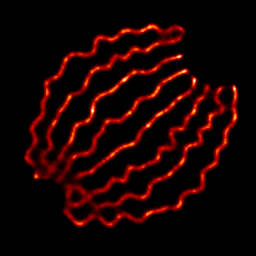}
  \end{subfigure}%
  
  \caption{{From left to right: High-resolution ground truth, Mean of original frames, Mean of reconstructed frames using \textbf{Algorithm~\ref{algo:poly-fista}} with backtracking line-search.}}
  \label{fig:frame1}
\end{figure}


\section{Conclusion}\label{sec:conclude}
\setcounter{equation}{0}

In this paper, we introduce a Newton-type method for nonsmooth composite optimization that leverages tilt-stability and effective subspace identification to achieve local quadratic convergence. Our approach departs from existing manifold Newton methods and can be viewed as a particular case of coderivative-based and SCD Newton methods, while avoiding the computation of sophisticated second-order structures. In particular, we provide constructive characterizations of the effective subspaces for a wide family of polyhedral regularizers, including the $\ell_1$ norm, the $\ell_\infty$ norm, the sorted $\ell_1$ norm, and the 1D total variation semi-norm. Numerical experiments illustrate the practical behavior of the proposed framework, showing that the effective-subspace update can substantially accelerate the corresponding first-order schemes on the tested instances and can perform favorably relative to several nonsmooth Newton-type methods. However, additional benchmark comparisons also indicate that, on classical problems such as Lasso and TV-1D, highly optimized solvers such as coordinate descent, Celer, \texttt{skglm}, and \texttt{sklearn} remain stronger in wall-clock performance.

Several directions remain open for future research. Extending our approach to non-polyhedral structured regularizers such as the  nuclear norm, group-sparsity with overlapping blocks, and the isotropic total variation would broaden the scope of applicability. Although the technique in our main Theorem~\ref{theo:quad-poly} heavily relies on properties of polyhedral mapping, we believe that the strategy used in this paper can handle the  non-polyhedral cases with some modifications. Moreover, our approach can be extended to solving the generalized equation in the format
\begin{equation}\label{eq:GE}
    0\in F(x) + \partial g(x),
\end{equation}
in which $F:\XX\to \XX$ is a continuously differentiable (not necessarily monotone) mapping and $g$ is a polyhedral mapping. In this case, the tilt-stability assumption at the minimizer $\ox$ of \eqref{p:CP} can be relaxed and replaced with the  {\em metric regularity} of the mapping $(F+\partial g)$ at a solution $\ox$ of \eqref{eq:GE} for $0$. This is a research in progress with potential applications to different and broader classes of optimization problems. 
\section*{Acknowledgments}
We are deeply grateful to C. Wylie for kindly sharing his code implementation of the Newton method with active manifolds combined with FISTA, as developed in \cite{wylie2019}, during the development of this paper. {We are also thankful to both anonymous referees for their careful readings, constructive suggestions, as well as deep remarks on other works that helped improve the original presentation significantly.}
\small

\bibliographystyle{abbrv}
\bibliography{bio}
\end{document}